\documentclass[12pt]{article}
\newcommand{\ii}{\mathrm{i}}

\newcommand{\e}{\mathrm{e}}
\newcommand{\St}{\mathcal{S}} 
\newcommand{\Gt}{\mathcal{G}} 
\newcommand{\Mt}{\mathcal{M}} 
\usepackage{graphicx}
\usepackage{amsmath}
\usepackage{amsfonts}
\usepackage{amsthm}
\usepackage{amssymb}
\usepackage{mathtools} 
\newtheorem{theorem}{Theorem}[section]
\newtheorem{lemma}[theorem]{Lemma}

\newtheorem{proposition}{Proposition}[section]
\newtheorem{corollary}[proposition]{Corollary}
\numberwithin{equation}{section}
 \usepackage{enumitem}
 \usepackage{color}
 \usepackage[bookmarks=true,colorlinks=true, pdfstartview=FitV, linkcolor=black, citecolor=blue, urlcolor=black]{hyperref}

\newcommand{\pFq}[5]{{}_{#1}F_{#2} \left( \genfrac{}{}{0pt}{}{#3}{#4}; #5 \right)  }
\title{Ratio and limiting zero distribution asymptotics for symmetric multiple orthogonal polynomials}
\author{Ana Loureiro\footnote{School of Engineering, Mathematics and Physics, University of Kent -- London and South East University Group, Canterbury
CT2 7NF, UK, \texttt{a.loureiro@kent.ac.uk}} \and Walter Van Assche\footnote{Department of Mathematics, KU Leuven, Celestijnenlaan 200B box 2400, BE-3001 Leuven, Belgium, \texttt{walter.vanassche@kuleuven.be}}}

\date{\today}

\begin{document}
\maketitle

\begin{abstract}
We investigate the ratio asymptotics and the asymptotic zero distribution of a sequence of polynomials that satisfy a recurrence relation of order $r+1$ with all recurrence coefficients, except the last one, equal to zero. Such a sequence is part of a system of multiple orthogonal polynomials and  it satisfies the symmetry property $P_n(\omega_{r+1} z) = \omega_{r+1}^n P_n(z)$, where $\omega_{r+1}$ is the
primitive $(r+1)$th root of unity.  
We consider the unbounded regime in which the recurrence coefficients exhibit algebraic growth and, after division by $n^{\gamma}$ become asymptotically periodic and bounded.
After the appropriate scaling, we establish ratio asymptotics and characterize the limiting ratio as the distinguished solution of an algebraic equation. We then determine the limiting zero distribution through its Stieltjes transform and investigate the associated \(\St\)-transform, which in several cases yields connections with hypergeometric polynomial sequences and distributions arising in free probability. The recurrence is represented by a two-diagonal non-self-adjoint Hessenberg operator, so that the limiting zero measure also admits a natural interpretation as a limiting empirical spectral distribution of its rescaled finite sections. Our analysis is based solely on the positivity and asymptotic behavior of the recurrence coefficients and requires no explicit knowledge of the underlying orthogonality measures.
\end{abstract}

\noindent{\bf Keywords.} Multiple orthogonal polynomials, banded Hessenberg matrix, ratio asymptotics, asymptotic limiting zero distribution, free probability

\tableofcontents

\section{Introduction}
We start with the recurrence relation
\begin{equation}  \label{eq:1}
   P_{n+1}(z) = zP_n(z) - \gamma_{n-r} P_{n-r}(z), \quad n\geq r, 
\end{equation}
where $r$ is an arbitrary integer with $r\geq 1$,  $(\gamma_n)_n$ is a sequence of positive real numbers and 
\begin{equation}\label{eq:1b}
P_0=1,\  P_{-1} = \cdots = P_{-r}=0 .
\end{equation} This gives a sequence
of monic polynomials, with $P_n$ of degree $n$, and these polynomials satisfy the symmetry property $P_n(\omega_{r+1} z) = \omega_{r+1}^n P_n(z)$, where $\omega_{r+1}=e^{\frac{2\pi i}{r+1}}$, which can easily be proved by induction. The zeros of the polynomial $P_n(z)$ are all on the $(r+1)$-star 
\begin{equation} \label{Sigma}
\Sigma = \bigcup_{k=0}^r \{ x \omega_{r+1}^k:\  x \geq 0\}.  
\end{equation}
In this paper we will assume that the recurrence coefficients are unbounded, but after scaling they become asymptotically periodic with period $r$, i.e.,
\begin{equation}   \label{eq:asper}
      \lim_{n \to \infty}  \frac{\gamma_n}{n^\gamma} = \alpha_k, \qquad  n \equiv (k-1) \bmod r,
\end{equation} 
with $\gamma >0$ and $\alpha_1,\ldots,\alpha_r \in \mathbb{R}^+$. If we put 
\begin{equation}\label{Q}
	Q_n(z,N) = N^{-\frac{n\gamma}{r+1}} P_n(N^{\frac{\gamma}{r+1}} z),
\end{equation}
then the recurrence \eqref{eq:1} becomes
\begin{equation}  \label{eq:2}
   Q_{n+1}(z,N) = zQ_n(z,N) - \frac{\gamma_{n-r}}{N^\gamma} Q_{n-r}(z,N).
\end{equation}
Our goal is to describe the ratio asymptotics of the polynomial sequence \((P_n(z))_{n\geq 0}\) as well as its asymptotic distribution, under the assumptions regarding the behavior of the recurrence coefficients \(\gamma_n\) outlined above.

Our first main result is:

\begin{theorem}\label{Thm1} Suppose $P_n$ satisfies \eqref{eq:1} and the recurrence coefficients $\gamma_n$ are positive and have the asymptotic behavior given by \eqref{eq:asper}. Then, under the definition \eqref{Q}, we have 
\begin{equation}  \label{eq:ratio}
   \lim_{n \to \infty, n/N \to t} \frac{Q_n(z,N)}{Q_{n+r}(z,N)} = \Phi(z,t),
\end{equation}
uniformly on compact subsets $K$ of $ \mathbb{C}\backslash \Sigma$, where $\Phi$ is a solution of the algebraic equation 
\begin{equation}  \label{eq:6}
	1=\Phi(z,t) \prod_{j=1}^{r}\left(z-\alpha_j t^\gamma \Phi(z,t)\right).
\end{equation}
for which $z^r\Phi(z,t) \to 1$ as $z \to \infty$.
\end{theorem}

The function \(\Phi(z,t)\) in \eqref{eq:6} satisfies a number of properties, which are explained in Section \ref{Subsec: Phi}. Most significantly, we can write 
\[
    \Phi(z,t) = t^{-\gamma} z F(t^{-\gamma} z^{r+1})
\]
where \(F\) is a solution of the implicit equation 
\begin{equation}\label{eqF}
	1= y F\left(z^{r+1}\right)A_r(F\left(z^{r+1}\right))  \qquad \text{with} \quad A_r(u) =  \prod_{j=1}^r \left(1 -\alpha_j u \right), 
\end{equation}
for which \(yF(y) \to 1\) as \(y\to \infty\).

Our second main result describes the Stieltjes transform of the asymptotic distribution of the zeros of the rescaled polynomials of large degree. 

\begin{theorem}\label{Thm2} If   \( \Gt (z) \) is the Stieltjes transform of the zero limiting distribution 
of the rescaled polynomial sequence \((Q_n(z))_{n\geq 0} \), then one has  
\begin{equation}   \label{eq:13}
  \Gt (z) 
   = - \frac{1}{r} \int_0^1 \frac{\displaystyle \partial_z \Phi(z,t) }{\Phi(z,t)} \, \mathrm{d}t , 
\end{equation}
uniformly on compact sets of $\mathbb{C} \setminus \Sigma$, where \(\Phi(z,t)\) is given by  \eqref{eq:ratio}-\eqref{eq:6} and such that $z^r\Phi(z,t) \to 1$ as $z \to \infty$. Moreover, we have 
\begin{align}
	\Gt (z) 
	&= - \frac{1}{r\,z}  
	\left( 1- \frac{(r+1)}{\gamma} \frac{1}{\left(F(y) A_r(F(y))\right)^{\frac{1}{\gamma}}} \int_0^{F(y)}  \left(u \, A_r(u)\right) ^{1/\gamma}	\frac{1}{u}\mathrm{d}u  \right), \label{G_integral1}
\end{align}
with  \(y=z^{r+1}\), with  \(A_r\) and \(F\) as in \eqref{eqF}. 
\end{theorem}

The density for the asymptotic zero distribution can be obtained via the inverse Stieltjes transform of $\Gt$. Since \(\Gt\) is analytic in \(\mathbb{C}\backslash \mathbb{R}\) and \(\Gt(z)=\mathcal{O}(1/z)\) as \(z\to +\infty\), then there exists a function \(\sigma\) defined on \(\mathbb{R}\) that 
\[\Gt (z) = \frac{1}{2\pi} \int_{\mathbb{R}} \frac{\sigma(t)}{z-t} \mathrm{d} t\] 
and \(\sigma\) is the density function we want. It can be obtained via the Sokhotski-Plemelj formula, which means that 
\[
	\sigma(x) = \frac{1}{2\pi\ii} \left( \Gt_{+}(x) - \Gt_{-}(x) \right)
\]
where \(\Gt_{\pm}(x) \) corresponds to boundary values of the function  \(\Gt(z) \) when \(z\) approaches the real line from above and from below, precisely 
\[\Gt_{\pm}(x) = \lim_{\epsilon \to 0+} \Gt(x\pm\ii \epsilon) . \]

\subsection{Background and motivation}

The case where the coefficients are  bounded and all have the same asymptotic behavior was first studied in  \cite{AKVI}, where a connection with multiple orthogonality (on the step-line) was made. There it is referred to as vector orthogonality, whilst linking it to Hermite-Pad\'e approximation. A subsequent study \cite{AKS} considered the ratio and weak asymptotics for the case where the recurrence coefficients are an $\ell^{1}$-perturbation of a constant positive coefficient. For the asymptotically periodic and  bounded recurrence coefficients, weak and ratio asymptotics were subsequently obtained in \cite{DelLop}, together with a vector-equilibrium interpretation. Closely related asymptotic questions have also been investigated for multiple orthogonal  polynomials associated with Nikishin systems, in particular on star-like sets: see \cite{DelLop,LGLL1,LGMD} and  references therein. See also \cite{WVA} for ratio asymptotics of multiple orthogonal  polynomials under convergence of the nearest-neighbour recurrence coefficients. 

A central novelty of the present work is that the recurrence coefficients are allowed to be unbounded with algebraic growth and a periodic modulation in the leading term (as detailed in \eqref{eq:asper}). More precisely, we study the ratio asymptotics and the corresponding limiting asymptotic distribution of the zeros of the polynomial sequence $(P_n)_{n\geq 0}$ directly from the recurrence relation \eqref{eq:1}, based solely on the positivity and the asymptotic condition \eqref{eq:asper}. Thus, for some $\gamma>0$ and $\alpha_1,\ldots,\alpha_r>0$,  $\gamma_{m r + k-1} \sim \alpha_k \, (mr+k-1)^{\gamma},\quad k=1,2,\ldots, r$ as $m\to\infty$. In particular, our analysis does not require explicit knowledge of the orthogonality measures, nor any 
additional structural assumption on them, such as the Nikishin property. Our main results are stated in Theorem \ref{Thm1} and Theorem \ref{Thm2}. In Section \ref{Sec RA} we first establish the estimates and interlacing properties of the zeros needed to prove the ratio asymptotics. These asymptotics are  then used in Section \ref{sec:Asympt} to obtain the limiting zero distribution through its Stieltjes transform (here denoted as \(\Gt\)-transform). We further investigate the associated \(\St\)-transform, building on the finite-free convolution framework \cite{MF-M-P1} and the asymptotic analysis framework for generalized hypergeometric polynomials developed in \cite{MF-MP2}, identifying cases where our \(\St\)-transforms relate to those arising from hypergeometric polynomials outlined therein. Finally, in Section \ref{Sec:examples} we work out several examples in detail. The last example treats the case in which one residue class of the recurrence coefficients has a different leading asymptotic coefficient from the remaining $r-1$ classes. One of the families studied in detail in \cite[Case B]{AnaWVA} provides a particular case with $r=2$. For reference, the
Appendix collects the basic properties and transforms of the Marchenko-Pastur, Kesten-McKay and the Fuss-Catalan distributions that are used throughout the examples. 

There is also a natural spectral interpretation of these results. The polynomial \(P_n\) is the characteristic polynomial of the \(n\times n\) principal truncation \(\mathbf{H}_n\) in \eqref{Hn} of the two-diagonal Hessenberg operator associated with the recurrence \eqref{eq:1}. Consequently, the zeros of the rescaled polynomial \(Q_n\) are the eigenvalues of the corresponding rescaled Hessenberg matrix. Thus, from this point of view, Theorem~\ref{Thm2} describes the limiting empirical spectral distribution of a family of non-self-adjoint banded Hessenberg matrices with algebraically growing recurrence coefficients whose leading-order behavior is periodic. Equivalently, its Stieltjes transform is obtained as the limit of the normalized trace of the resolvents of these finite sections. 
With the rescaling \eqref{Q}, then we have 
\[
	Q_{n}(z;N) = \det \left(z \mathbf{I} - N^{-\gamma/(r+1)} \mathbf{H}_n\right)
\]
and, consequently, when $n=N$, 
\[ \frac{1}{N} \frac{Q_N'(z;N)}{Q_N(z;N)} = \frac{1}{N} \mathrm{Tr} \left(z \mathbf{I} - N^{-\gamma/(r+1)} \mathbf{H}_N\right)^{-1}\ .\]

These results also have a natural connection with spectral theory and random matrix models.  Multiple orthogonal polynomials occur as average characteristic polynomials and in the correlation structure of several random matrix models and related determinantal point processes.  For example, they are average characteristic polynomials of Hermitian random matrices with external source  \cite{BleherKuijlaars04}, two-matrix models \cite{DuitsKuijlMo} and non-intersecting path models \cite{DaemsKuijl}.

As such, Theorem~\ref{Thm2} suggests a potentially useful framework: whenever the average characteristic polynomial of the relevant multiple orthogonal polynomials of a model have recurrence coefficients satisfying the asymptotic condition $\gamma_n \sim \alpha_k \, n^\gamma$, the global zero asymptotics of the associated average characteristic polynomials could potentially be obtained directly from their recurrence coefficients, without first carrying out a full Riemann-Hilbert or equilibrium analysis. In models where the limiting zero distribution of the average characteristic polynomials is known to coincide with the limiting eigenvalue distribution, this may also provide an alternative route to the global spectral law.

In special cases, the limiting measures arising here
coincide with familiar distributions from random matrix theory and free
probability. For instance, the Fuss-Catalan law occurring in
Section \ref{Sec:examples} is also the limiting squared singular-value distribution
for products of Ginibre matrices. The $\St$-transform therefore provides a
useful tool for comparing the spectral laws arising from the recurrence
\eqref{eq:1} with those coming from random matrices and from hypergeometric
polynomial ensembles.

\section{Ratio asymptotics}\label{Sec RA}

We begin by showing that the zeros of each of the polynomials  $P_n$ satisfying \eqref{eq:1} with $\gamma_n>0$ 
enjoy interlacing properties, rotational symmetry and are all lying on the $(r+1)$-star $\Sigma $ in \eqref{Sigma}. These structural properties are fundamental to the study of ratio asymptotics. Accordingly, in Section \ref{Subsec Proof Thm1} we prove our first main result -- Theorem \ref{Thm1} -- and in Section \ref{Subsec: Phi} we analyze the relevant properties of the function \(\Phi\) describing the ratio asymptotics.

\subsection{Properties of the zeros}

The asymptotic behavior of the $\gamma$-coefficients is closely related to that of the largest zero of each polynomial. For the case of two measures $(r=2)$ this has been shown in \cite[Th. 2.6]{AnaWVA}, whose proof can be straightforwardly extended to the case of $r\geq 2$ measures and, in this case, we have:

\begin{proposition}\label{prop:max zero P} Suppose $P_n(x)$ is a polynomial of degree $n$ satisfying \eqref{eq:1} with $\gamma_n>0$ 
and that \eqref{eq:asper} holds. If $\alpha=\max\limits_{k\in\mathbb{Z}_r}\{\alpha_k\}$, then 
the largest zero in absolute value $|x_{n,n}|$ of $P_n(x)$ satisfies
\begin{equation}
	|x_{n,n}| \leq  \left( 1 + \frac{1}{r}\right) \left( \alpha\  r \ n^{\gamma} \right)^{1/(r+1)}+ o(n^{\gamma/(r+1)}) \quad \text{as}\quad n\to+\infty.
\end{equation}
\end{proposition}
\begin{proof} Consider the Hessenberg matrix 
\begin{equation}\label{Hn}
	\mathbf{H}_n = \left( 
		\begin{array}{ccccccccc}
			0 		 & 1 			& 0 &  \cdots & 0 & 0 & 0 & 0 & 0\\
			0 		 & 0 			& 1 & 0 & \cdots & 0 & 0 & 0 & 0\\
			\vdots & \vdots & \\
			0	&   &   & \ddots \\
			\gamma_1 & 0 			&  & \cdots &  1 & \cdots &  & 0 & 0 \\
			0 		 & \gamma_2 	& 0 &  & \cdots & 1 & \cdots & 0 & 0 \\
			    		 &   &   \ddots&    &   &      &   \ddots&  & \\
					&   &  &   \ddots &   &   &      & \ddots  &  \\
			   		0 & 0  & 0 & \cdots &  \gamma_{n-r-1}  & \cdots &   0&   0& 1 \\
			   		 0 &   0&  0&0  & \cdots &   \gamma_{n-r}  &   \cdots &   0& 0  \\
		\end{array}
	\right)
\end{equation}
so that the recurrence relation \eqref{eq:1} can be expressed as 
\[
	\mathbf{H}_n \left(\begin{array}{c} P_0(x) \\ P_1(x) \\ \vdots \\ P_{n-1}(x)\end{array}\right)
	= x \left(\begin{array}{c} P_0(x) \\ P_1(x) \\ \vdots \\ P_{n-1}(x)\end{array}\right)
	- P_n(x) \left(\begin{array}{c} 0 \\ 0\\ \vdots \\ 1\end{array}\right)
\]
and each zero of $P_n$ is an eigenvalue of the matrix $\mathbf{H}_n$. 
The spectral radius of the matrix $\mathbf{H}_n$, 
\[
	\rho(\mathbf{H}_n) = \max \{ |\lambda| : \ \lambda \text{ is an eigenvalue of } \mathbf{H}_n\}, 
\]
is bounded from above by  $|| \mathbf{H}_n ||$ where $|| \cdot ||$ denotes a matrix norm (see \cite[Section 5.6]{HJ}). 
We take the matrix norm 
\[
	\left\Vert \mathbf{H}_n \right\Vert_S =|| S^{-1}  \mathbf{H}_n S ||_{\infty} = \max_{1\leq i\leq n} \left\{ \sum_{j=1}^n \left|(S^{-1} \mathbf{H}_n S)_{i,j} \right| \right\}, 
\]
where $S$ corresponds to a non-singular matrix and $(S^{-1} \mathbf{H}_n S)_{i,j}$ denotes the $i$th row and $j$th column entry of the product matrix $S^{-1} \mathbf{H}_n S$. In particular if $S$ is an invertible diagonal matrix   $S = \text{diag}(d_1,\ldots, d_k, \ldots, d_{n})$, then 
\[
\begin{multlined}
	\left\Vert \mathbf{H}_n \right\Vert_S \\
	= \max \left\{ \tfrac{d_2}{d_1}, \ldots ,  \tfrac{d_{r+1}}{d_r},    \tfrac{d_{r+2} + d_1 \gamma_1}{d_{r+1}}, \ldots, \tfrac{d_j + d_{j-r-1} \gamma_{j-r-1}}{d_{j-1}} , \ldots, 
	\tfrac{d_n + d_{n-r-1} \gamma_{n-r-1}}{d_{n-1}},  \tfrac{ d_{n-r} \gamma_{n-r}}{d_{n}} \right\}. 
\end{multlined}
\]
Setting $d_k= d^{k} (k!)^{\frac{\gamma}{r+1}}\neq 0$, for some  positive constant $d$, gives 
\[
	\left\Vert \mathbf{H}_n \right\Vert_S \leq \left(d+\frac{\alpha}{d^{r}}\right) n^{\gamma/(r+1)} + o(n^{\gamma/(r+1)}) \quad \text{as }\ n\to +\infty. 
\]
The choice of $d=\left(\alpha r\right)^{1/(r+1)}$ gives a minimum to $\left(d+\frac{\alpha}{d^{r}}\right)$, so that  
\[
	\left\Vert \mathbf{H}_n \right\Vert_S \leq  \left( 1 + \frac{1}{r}\right) \left( \alpha\  r \ n^{\gamma} \right)^{1/(r+1)}+ o(n^{\gamma/(r+1)}) \quad \text{as }\ n\to +\infty, 
\]
which implies the result. 
\end{proof}

Except at the origin, all the zeros of each $P_n$ satisfying \eqref{eq:1} are simple and interlace with those of $P_{n+1},\ldots, P_{n+r-1}$, as formally stated next. \\

\begin{proposition}[{\it c.f.} {\cite[Th. 2.2]{BR08}}]
\label{prop:zeros}
Suppose $P_n$ satisfies \eqref{eq:1} with $\gamma_n>0$. For each $n=(r+1)m+j$ with $m\geq0$ and $j\in\mathbb{Z}_{r+1}$, the following statements hold: 
\begin{enumerate}[label={\normalfont(\alph*)}]
\item If $z$ is a zero of $P_n$, then so is $\omega^j z$ for any $j\in\mathbb{Z}_{r+1}$, where $\omega=\exp(2\pi {\rm i}/(r+1))$. 
\item $P_n$ has a zero of multiplicity $j$ at the origin  and another $m$ distinct positive real zeros, which we denote by $x_{n,j,s}$ with $s\in\{1,\ldots , m\}$ and therefore we order them as
\[	
	0< x_{n,j,1} < x_{n,j,2} <\ldots < x_{n,j,m}. 
\]
\item The positive zeros $x_{n,j,s}$ interlace: 
\[
	0 < x_{n,0,1}<\ldots < x_{n,r,1} < x_{n,0,2} < \ldots < x_{n,j,s}<x_{n,j+1,s} < \ldots < x_{n,r-1,m}<x_{n,r,m} . 
\]
\end{enumerate}

\end{proposition}

The interlacing property for the real zeros of $Q_{n}(z;N)$ follows straightforwardly from the latter result.

\begin{corollary}\label{cor:zero bounds} Let $Q_n(z;N)$ satisfy \eqref{eq:2} with $\gamma_{n-r}>0$ and suppose 
\begin{equation}\label{gamma lim}
	\lim_{n\to \infty, n/N\to t} \frac{\gamma_{n-r}}{N^\gamma} =  \alpha_k t^{\gamma}, \qquad  n \equiv k \bmod r. 
\end{equation}
If $\ n=m (r+1)+j$, with $j\in\mathbb{Z}_{r+1}$, then $Q_{n}(z)$ has $m$ positive real zeros $y_{n,j,s}$ and can be written as 
\begin{equation}\label{Qn factor}
	Q_n(z;N) = z^{j} \prod_{s=1}^m \left( z^{r+1} - (y_{n,j,s})^{r+1}  \right), \quad\text{with} \quad
	y_{n,j,s}(N)= N^{-\frac{\gamma}{r+1}}x_{n,j,s} \ . 
\end{equation}
 Moreover, all the zeros interlace and the largest zero 
$y_{n,j,m}$  satisfies 
\[
	|y_{n,j,m}(N)| \leq  \left( 1 + \frac{1}{r}\right) \left( \alpha\  r \ t^{\gamma} \right)^{1/(r+1)}+ o\left(t^{\gamma/(r+1)}\right) 
	\quad \text{as}\quad n\to+\infty, \ \text{and} \  \frac{n}{N}\to t, 
\]
where $\alpha=\max\limits_{k\in\mathbb{Z}_r} \alpha_k$. 
\end{corollary}
\begin{proof}
According to Proposition \ref{prop:zeros}, if  $P_n$ satisfies \eqref{eq:1} with $\gamma_n>0$, then it can be written as 
\[ 
	P_n(z) = z^{j} \prod_{s=1}^m \left( z^{r+1} - (x_{n,j,s})^{r+1}  \right), \quad n = (r+1)m+j, \ m\geq 0, \ j\in\mathbb{Z}_{r+1}. 
\]
Following the definition $Q_n(z;N)$ in \eqref{Q}, we obtain \eqref{Qn factor}. The interlacing property of the zeros follows directly from Proposition \ref{prop:zeros} and the majorization result for the largest zero in absolute value follows from Proposition \ref{prop:max zero P}. 
\end{proof}

\subsection{Proof of Theorem \ref{Thm1}}\label{Subsec Proof Thm1}

Suppose that \begin{equation}   \label{eq:3}
     \lim_{n \to \infty, n/N \to t}  \frac{Q_n(z,N)}{Q_{n+1}(z,N)} = \phi_k(z,t), \qquad n \equiv ( k-1) \bmod r,
\end{equation} holds, for \(k=1,2,\ldots,r\). Then dividing \eqref{eq:2} by $Q_{n+1}(z,N)$ gives
\begin{equation}\label{rec rel2}
 1 = z \frac{Q_n(z,N)}{Q_{n+1}(z,N)} - \frac{\gamma_{n-r}}{n^\gamma} \left(\frac{n}{N}\right)^\gamma \frac{Q_{n-r}(z,N)}{Q_{n+1}(z,N)}, 
\end{equation}
and if $n \equiv k\bmod r$ and $n \to \infty$ and $n/N \to t$, one finds
\begin{equation}  \label{eq:4}
     1 = z \phi_k(z,t) - \alpha_k t^\gamma \phi_k(z,t) \Phi(z,t),
\end{equation}
where
\[      \Phi(z,t) = \phi_1(z,t)\phi_2(z,t)\cdots \phi_{r}(z,t).   \]
From \eqref{eq:4} we find
\begin{equation}  \label{eq:5}
    \frac{1}{\phi_k(z,t)} = z - \alpha_k t^\gamma \Phi(z,t), \qquad 1 \leq k \leq r. 
\end{equation}
Multiplying these equations then gives \eqref{eq:6}. 
If we now define
\[   B_r(x) = \prod_{j=1}^{r} (x-\alpha_j), \]
then we find that $\Phi$ satisfies the equation
\begin{equation}   \label{eq:7}
   \frac{1}{\Phi(z,t)} = t^{r\gamma} \Phi^r(z,t)\  B_r\left(\frac{zt^{-\gamma}}{\Phi(z,t)} \right) .
\end{equation}

We still need to prove that there exist $r$ functions $\phi_1(z,t),\phi_2(z,t),\ldots,\phi_r(z,t)$ such that \eqref{eq:3} holds uniformly on compact sets of $\mathbb{C} \setminus \Sigma$.

We first prove that $ \left| \frac{Q_n(z;N)}{Q_{n+1}(z,N)} \right|$ is bounded for all $z$ on a compact set in $\mathbb{C} \setminus \Sigma$. 

\begin{lemma}\label{lem: ratio1} Let $K$ be a compact set in $\mathbb{C} \setminus \Sigma$ and $\delta$ the distance between $K$ and $\Sigma$:
\[   \delta = \min  \{ |x-y|, x \in K, y \in \Sigma \}.  \] 
Suppose $Q_n$ satisfies \eqref{eq:2} with $\gamma_n>0$ subject to \eqref{gamma lim}. Then,  if $y_{n,j,s}$ represent the positive zeros of $Q_n$, there exist positive $c_1(t), \ c_2(t)$ and  $M(t)$ such that 
\[ 
	c_1(t) \leq \frac{\gamma_{n-r}} {N^\gamma} \leq c_2 (t)
	\qquad \text{and}\qquad 
	\left| y_{n,j,s}\right|\leq M(t), 
\]
 and the following inequalities hold: 
\begin{enumerate}
\item[(a)]
\begin{align}
& \label{ratio r Qn}
		\left|\frac{Q_{n-r}(z,N)}{Q_{n+1}(z,N)}  \right|
	\leq  \frac{1}{\delta}
	\ \text{and}\ 
	 \left|\frac{ Q_n(z;N)}{Q_{n+1}(z,N)} \right|
	 \leq \frac{1}{\delta} \left( 1 + \frac{c_2(t)}{\delta}\right),  \quad\text{for all} \ z \in K. 
\end{align}

\item[(b)] 
\[
	\left| \frac{Q_{n-r}(z,N)}{Q_{n+1}(z,N)}  \right| > \frac{1}{2|z|^{r+1}} , \]
and 
\[ \left|\frac{ Q_n(z;N)}{Q_{n+1}(z,N)} \right| > \frac{1}{|z|}\left(   \frac{c_1(t)}{2\, |z|^{r+1}} -1\right) \]
for all $|z| > M(t)$.

\end{enumerate}
\end{lemma}
 
\begin{proof}  
Let $K$ be a compact set in $\mathbb{C} \setminus \Sigma$ and $\delta$ the distance between $K$ and $\Sigma$:
\[   \delta = \min  \{ |x-y|, x \in K, y \in \Sigma \}.  \] 
Our strategy is to prove that 
\[  
 \left| \frac{Q_n(z;N)}{Q_{n+1}(z,N)} \right| \leq \delta^* , \qquad z \in K.  
\] 
The $(r+1)$-fold symmetry of the sequence $\{Q_n(z,N)\}$ means that there are $(r+1)$ sequences $\{Q_m^{[k]}(z,N)\}_{n\geq0}$ with $\deg Q_m^{[k]}(z,N) =m$ for each $k\in\mathbb{Z}_{r+1}$ such that 
\[
	Q_n(z,N) = z^k Q_m^{[k]}(z^{r+1},N), \quad \text{where } \quad n=(r+1)m+k. 
\]
Therefore, for any $n\geq 0$, the ratio 
\[ \frac{ z Q_n(z;N)}{Q_{n+1}(z,N)}\] 
only depends on $z^{p+1}$ and no other powers of $z$ are involved. Proposition \ref{prop:zeros} tells us that $Q_{m(r+1)+k}(z;N)$ has a zero of order $j$ at the origin, $m$ distinct zeros on the positive real line, and all the other $mr$ zeros are $\frac{2\pi}{r+1}$ rotations of the zeros on the positive real line. Precisely, if $z_{n,N,j}\neq 0$, for $0\leq j \leq m(r+1)$ is a zero of $Q_{m(r+1)+k}(z;N)$ we then have 
\[
	z_{n,N,j} = (y_{n,N,\ell})^{r+1} \quad \text{with} \quad   y_{n,N,\ell}\in \mathbb{R}_+ 
	\quad \text{and} \quad  0\leq \ell \leq m.
\]
This means we can consider the  partial fractions decomposition 
\begin{equation}\label{p frac Q}
	  \frac{ z Q_n(z;N)}{Q_{n+1}(z,N)} 
	  = 1+\sum_{\ell=1}^m \frac{A_{n+1,N,\ell}}{z^{r+1}-\left(y_{n+1,N,\ell}\right)^{r+1}}  ,  
\end{equation}
where $(n+1-m(r+1)) \in\mathbb{Z}_{r+1}$. The interlacing property of the zeros (described in Proposition  \ref{prop:zeros}) implies that $A_{n+1,N,\ell}$ in \eqref{p frac Q} are positive. On the other hand, from the recurrence relation \eqref{eq:2} we have 
\[
	 \frac{\gamma_{n-r}}{N^\gamma} \frac{Q_{n-r}(z,N)}{Q_{n+1}(z,N)}  
	 =  \sum_{\ell=1}^m \frac{A_{n+1,N,\ell}}{z^{r+1}-\left(y_{n+1,N,\ell}\right)^{r+1}} . 
\]
We take the limit as $z\to+\infty $ on both sides of the latter identity after multiplying it by $z^{r+1}$, to then obtain 
\begin{equation}\label{ident sum As}
 	 \frac{\gamma_{n-r}}{N^\gamma}	 =  \sum_{\ell=1}^m A_{n+1,N,\ell} . 
\end{equation}
Therefore we obtain \eqref{ratio r Qn} because 
\[
	\left|\frac{Q_{n-r}(z,N)}{Q_{n+1}(z,N)}  \right|
	=\left|  \frac{N^\gamma} {\gamma_{n-r}}  \sum_{\ell=1}^m \frac{A_{n+1,N,\ell}}{z^{r+1}-\left(y_{n+1,N,\ell}\right)^{r+1}} \right|
	\leq \frac{1}{\delta}  \left| \frac{N^\gamma} {\gamma_{n-r}} \right|  \left| \sum_{\ell=1}^m A_{n+1,N,\ell} \right|
	= \frac{1}{\delta} .
\]
The boundedness of $\frac{\gamma_{n-r}}{N^\gamma}$ implies 
\[
	\left|  \sum_{\ell=1}^m A_{n+1,N,\ell} \right| \leq c_2.
\]
For $z\in K$, the left hand side of \eqref{p frac Q} satisfies 
\begin{equation}
	\delta  \left|\frac{ Q_n(z;N)}{Q_{n+1}(z,N)} \right| \leq   \left|\frac{z Q_n(z;N)}{Q_{n+1}(z,N)} \right| , 
\end{equation}
whilst the right hand side is such that 
\[
	\left|  1+\sum_{\ell=1}^m \frac{A_{n+1,N,\ell}}{z^{r+1}-\left(y_{n+1,N,\ell}\right)^{r+1}}  \right|
	\leq  1 + 	\frac{\left|  \sum_{\ell=1}^m A_{n+1,N,\ell} \right|}{\delta} 
	\leq 1+ \frac{c_2}{\delta}. 
\]
As a result, we obtain \begin{equation}\label{ratio bou}
	 \left|\frac{ Q_n(z;N)}{Q_{n+1}(z,N)} \right|
	 \leq \frac{1}{\delta} \left( 1 + \frac{c_2}{\delta}\right).  
\end{equation}

In order to prove part (b), we start by observing that if $|z|>M=\max\limits_{n\in\mathbb{N}} (y_{n,j,m})$, then $|y_{n,j,s}/z|<1$ so that $\mathfrak{R}\left( \frac{1}{1-y_{n,j,s}/z}\right)>\frac{1}{2}$ for 
$1\leq s\leq m$ and $j\in\mathbb{Z}_{r+1}$. Therefore, 
\begin{align*}
	 \left|\frac{Q_{n-r}(z,N)}{Q_{n+1}(z,N)}  \right|
	&=\left|  \frac{N^\gamma} {\gamma_{n-r}}  \sum_{\ell=1}^m \frac{A_{n+1,N,\ell}}{z^{r+1}-\left(y_{n+1,N,\ell}\right)^{r+1}} \right| \\
	&=\frac{N^\gamma} {\gamma_{n-r}} \frac{1}{|z|^{r+1}}  \left|  \sum_{\ell=1}^m \frac{A_{n+1,N,\ell}}{1-\left(y_{n+1,N,\ell}/z\right)^{r+1}} \right| \\
	&\geq \frac{N^\gamma} {\gamma_{n-r}} \frac{1}{|z|^{r+1}}   \sum_{\ell=1}^m\mathfrak{R}  \left(  \frac{A_{n+1,N,\ell}}{1-\left(y_{n+1,N,\ell}/z\right)^{r+1}} \right) \\
	 &> \frac{N^\gamma} {\gamma_{n-r}} \frac{1}{2\, |z|^{r+1}}   \sum_{\ell=1}^m  A_{n+1,N,\ell} 
	  =   \frac{1}{2\, |z|^{r+1}}, 
\end{align*}
where, for the last identity, we have used \eqref{ident sum As}. 

Finally, the second inequality in part (b) follows from 
\begin{align*}
	  \frac{1}{2\, |z|^{r+1}} &<  \left|\frac{Q_{n-r}(z,N)}{Q_{n+1}(z,N)}  \right| \\
    &=\left| \left(  \frac{z\, Q_{n}(z,N)}{Q_{n+1}(z,N)}  -1\right)  \frac{N^\gamma} {\gamma_{n-r}} \right| \\
	 &\leq \left( \left|  \frac{z\, Q_{n}(z,N)}{Q_{n+1}(z,N)}\right| +1 \right) \left| \frac{N^\gamma} {\gamma_{n-r}} \right|\\[0.2cm]
	&\leq \left( \left|  \frac{z\, Q_{n}(z,N)}{Q_{n+1}(z,N)}\right| +1 \right) \frac{1}{c_1} . 
\end{align*}
\end{proof}

We can then conclude that 
\[   \left( \frac{Q_{mr+k}(z)}{Q_{mr+k+1}(z;N)} \right)_m   \]
is a normal family on $\mathbb{C} \setminus \Sigma$. By Montel's theorem for $k=1$, there is a subsequence that converges to some function $\phi_1(z,t)$, uniformly
on compact sets of $\mathbb{C} \setminus \Sigma$.

 There is also a further subsequence that converges for $k=2$ to $\phi_2(z,t)$, and we can do this for every $k$ with $1 \leq k \leq r$. We conclude that there is a subsequence for which there are functions $\phi_1,\ldots,\phi_r$ such that
\begin{equation}  \label{eq:14}
   \lim_{m \to \infty, mr/N \to t} \frac{Q_{mr+k}(z,N)}{Q_{mr+k+1}(z,N)} = \phi_k(z,t), 
\end{equation}
uniformly on compact subsets of $\mathbb{C} \setminus \Sigma$. 

These $\phi_1,\ldots,\phi_r$ may depend on the subsequence, so we need to show that they do not depend on the subsequence. In order to do so, we use the following two auxiliary results. 

\begin{lemma}\label{lem:aux1}
Suppose $(a_n)_{n\geq 0}$ is a positive sequence for which
\begin{equation}  \label{ineq1}
    a_{n+1} \leq c_n + a \sum_{k=0}^{r-1} a_{n-k}, \qquad  n \geq 0, 
\end{equation}    
where $r$ is an integer $\geq 1$, $0 < a \leq 1$ and $(c_n)_{n \geq 0}$ a positive sequence. Then one has
\begin{equation}  \label{ineq2}
     a_{n+1} \leq   a_0 A^{n+1} + \sum_{k=0}^n c_k A^{n-k}, 
\end{equation}
where $A = 2 a^{1/r}$.
\end{lemma}

\begin{proof}
We will use induction on $n$. For $n=0$ we have from \eqref{ineq1}
\[    a_1 \leq c_0 + a a_0.  \]
Now use the fact that $a \leq a^{1/r} < 2a^{1/r}$ whenever $0 < a \leq 1$ to find that $a_1 \leq c_0+ a_0 A$, which is \eqref{ineq2} for $n=0$.

Now assume \eqref{ineq2} is true up to $n-1$. Then from \eqref{ineq1} we find
\[     a_{n+1} \leq  c_n + a \sum_{k=0}^{r-1} \left(  a_0 A^{n-k} + \sum_{j=0}^{n-k-1} c_j A^{n-k-j-1} \right).  \]
Now observe that 
\begin{equation}  \label{aA}
       a \sum_{k=0}^{r-1} A^{-k} = a\  \frac{\frac{1}{A^r}-1}{\frac{1}{A}-1} \leq A, \qquad  \textrm{if } 0 < a \leq 1. 
\end{equation}
Hence
\[    a  a_0  \sum_{k=0}^{r-1} A^{n-k} \leq a_0 A^{n+1}. \]
Furthermore
\[     a \sum_{k=0}^{r-1} \sum_{j=0}^{n-k-1} c_j A^{n-k-j-1}  = a \sum_{j=0}^{n-r} c_j A^{n-j-1}  \sum_{k=0}^{r-1} A^{-k} 
+ a \sum_{j=n-r+1}^{n-1} c_j A^{n-j-1} \sum_{k=0}^{n-j-1} A^{-k} .  \]
By using \eqref{aA} we see that for the first term on the right
\[   a \sum_{j=0}^{n-r} c_j A^{n-j-1}  \sum_{k=0}^{r-1} A^{-k} \leq  \sum_{j=0}^{n-r} c_j A^{n-j}.  \]
In the second sum on the right we have that $j \geq n-r+1$ so that
\[        \sum_{k=0}^{n-j-1} A^{-k} \leq  \sum_{k=0}^{r-1} A^{-k}, \]  
and hence, by using \eqref{aA} again,
\[     a \sum_{j=n-r+1}^{n-1} c_j A^{n-j-1} \sum_{k=0}^{n-j-1} A^{-k} \leq  \sum_{j=n-r+1}^{n-1}  c_j A^{n-j}.  \]
Combining this gives the required inequality \eqref{ineq2}.
\end{proof}

\begin{corollary}\label{cor:aux1}
Suppose $a < 2^{-r}$ and $(c_n)_{n \geq 0}$ is a positive sequence for which $\lim_{n \to \infty} c_n = 0$. If
\[    0 \leq  a_{n+1} \leq c_n + a \sum_{k=0}^{r-1} a_{n-k}, \qquad  n \geq 0,  \]
then $\lim_{n \to \infty} a_n= 0$.  
\end{corollary}

\begin{proof}
We use the previous lemma, which gives
\[   0 \leq a_{n+1} \leq a_0 A^{n+1} + \sum_{k=0}^n c_k A^{n-k} ,  \]
with $A < 1$. Then the result follows easily because
\[   \lim_{n \to \infty} \sum_{k=0}^n c_k A^{n-k} = 0 .  \]
\end{proof}

We use the two last results to prove the following: 
\begin{lemma}\label{lemma independence of limits} Under the same assumptions as in Lemma \ref{lem: ratio1}, then  
\[ 
	\lim_{n\to \infty, mr/N \to t}\left| \frac{Q_{n}(z;N)}{Q_{n+1}(z;N)} - \frac{Q_{n+r}(z;N)}{Q_{n+1+r}(z;N)} \right| = 0, 
\] 
uniformly on compact subsets $K\subset \mathbb{C} \backslash \Sigma$.
\end{lemma}
\begin{proof}
From the recurrence relation \eqref{eq:2}, we have 
\[
	z = \frac{Q_{n+1}(z;N)}{Q_{n}(z;N)} 
		+ \frac{\gamma_{n-r}}{N^\gamma} \frac{Q_{n-r}(z;N)}{Q_{n}(z;N)} . 
\]
The same relation with $n$ replaced by $n+r$ reads as 
\[
	z = \frac{Q_{n+r+1}(z;N)}{Q_{n+r}(z;N)} 
		+ \frac{\gamma_{n}}{N^\gamma} \frac{Q_{n}(z;N)}{Q_{n+r}(z;N)} . 
\]
Subtracting the former from the latter leads to
\begin{align*}
&	 \frac{Q_{n+1}(z;N)}{Q_{n}(z;N)}  - \frac{Q_{n+r+1}(z;N)}{Q_{n+r}(z;N)} \\
&	 =  \frac{\gamma_{n}}{N^\gamma} \left(  \frac{Q_{n}(z;N)}{Q_{n+r}(z;N)} - \frac{Q_{n-r}(z;N)}{Q_{n}(z;N)}\right) 
	 + \left(   \frac{\gamma_{n}}{N^\gamma} - \frac{\gamma_{n-r}}{N^\gamma}  \right) \frac{Q_{n-r}(z;N)}{Q_{n}(z;N)}
\end{align*}
which is the same as 
\begin{align*}
	& \frac{Q_{n+r+1}(z;N)}{Q_{n}(z;N)} \left( \frac{Q_{n+1}(z;N)}{Q_{n+r+1}(z;N)}  -   \frac{Q_{n}(z;N)}{Q_{n+r}(z;N)}  \right) \\
	 &  = \frac{\gamma_{n}}{N^\gamma} \left(  \frac{Q_{n}(z;N)}{Q_{n+r}(z;N)} - \frac{Q_{n-r}(z;N)}{Q_{n}(z;N)}\right) 
	 + \left(   \frac{\gamma_{n}}{N^\gamma} - \frac{\gamma_{n-r}}{N^\gamma}  \right) \frac{Q_{n-r}(z;N)}{Q_{n}(z;N)}. 
\end{align*}
If we set 
\[
	D_{n}(z;N) = \frac{Q_{n}(z;N)}{Q_{n+r}(z;N)}  -   \frac{Q_{n-1}(z;N)}{Q_{n-1+r}(z;N)} , 
\]
then we obtain the following $r$th order recurrence relation for $D_n(z;N)$
\begin{align*}
	&D_{n+1} (z;N)\\
	 &  =   \frac{Q_{n}(z;N)}{Q_{n+r+1}(z;N)}\frac{\gamma_{n}}{N^\gamma} \left(  \sum_{\ell = 0}^{r-1} D_{n-\ell} (z;N) \right) 
	 + \left(   \frac{\gamma_{n}}{N^\gamma} - \frac{\gamma_{n-r}}{N^\gamma}  \right) \frac{Q_{n-r}(z;N)}{Q_{n+r+1}(z;N)}, 
\end{align*}
subject to the initial conditions $D_j(z;N)=0$ for negative $j$ and $D_0=z^{-r}$.
\[ D_j(z;N) = z^{j-1}  \left( \frac{z}{Q_{j+r}(z;N)}  -   \frac{1}{Q_{j-1+r}(z;N)}\right) = \frac{\gamma_j }{N^\gamma }\frac{z^{2j-1}}{Q_{j-1+r}(z;N)Q_{j+r}(z;N)}  \]
 for $j=1,\ldots , r$. 

Using \eqref{ratio r Qn}, it follows 
\begin{align}\label{Dn ineq}
	\left| D_{n+1}(z;N) \right| 
		& \leq
		C_n+  \frac{c_2(t)}{\delta} \sum_{\ell = 0} ^{r-1}  \left| D_{n-\ell} (z;N) \right| , 
\end{align}
with 
\begin{equation*}
	C_{n} =\left|  \frac{\gamma_{n}}{N^\gamma} - \frac{\gamma_{n-r}}{N^\gamma}  \right| 
				\frac{1}{\delta^{r+1}}  \left( 1 + \frac{c_2(t)}{\delta}\right)^r, \ n\geq 0. 
\end{equation*}
Under the assumption that \eqref{gamma lim} holds uniformly, then this implies that for every $\delta>0$
\[
	\lim_{\substack{n\to\infty\\ n/N\to t}} C_n =0. 
\]
If we choose a compact set $K^\star$ so that $\delta$ is large enough in order to ensure that $ \frac{c_2(t)}{\delta} < 2^{-r} $. Therefore Lemma \ref{lem:aux1} and Corollary \ref{cor:aux1} can be applied to conclude that 
\[ 
|D_{n+1}(z;N)|\to 0 \quad \text{as}\quad n\to \infty
\] 
uniformly for $z\in K^\star$. 

Besides, 
\[
	|D_{n+1} | 
		\leq 2 \ \left(\frac{1}{\delta} \left( 1 + \frac{c_2(t)}{\delta}\right)\right)^{r}
\]
on any compact $K\subset \mathbb{C}\backslash \Sigma$. By Vitali's theorem, we can conclude that $|D_{n+1}|$ converges uniformly to $0$ on every compact set $K\subset \mathbb{C}\backslash \Sigma$. 
\end{proof}

Hence, we conclude that $Q_{mr+k}(z,N)/Q_{mr+k+1}(z,N)$ converges to $\phi_k$ along a subsequence $m=m_\ell$, $\ell \to \infty$,
then also
\[  \lim_{m=m_\ell \to \infty} \frac{Q_{(m+1)r+k}(z,N)}{Q_{(m+1)r+k+1}(z,N)} = \phi_k(z,t).  \]

From here we can do the reasoning leading to \eqref{eq:3}--\eqref{eq:4}, and hence $\Phi=\phi_1\cdots\phi_r$
is the solution of \eqref{eq:7} which behaves like $1/z^r$ as $z \to \infty$, and thus every convergent subsequence has the same limits $\phi_1,\ldots,\phi_r$, which are obtained from $\Phi$ by using \eqref{eq:5}. That implies that the full sequence converges and \eqref{eq:3} holds.

\subsection{Properties of the ratio asymptotics}\label{Subsec: Phi}

Observe that \eqref{eq:7} is an algebraic equation of order $r+1$ for the function $\Phi$. 
We need the solution for which $\Phi(z,t) \sim z^{-r}$ as $z \to \infty$.
The equation also has $r$ other solutions for which $z/\Phi \sim \alpha_j t^\gamma$ as $z \to \infty$, for $1 \leq j \leq r$. 
Once we know what $\Phi$ is, we can find every $\phi_k$ from \eqref{eq:5}.
The function $\Phi$ depends on two variables $z$ and $t$. We can remove the dependence on $t$ by the following property.

\begin{proposition}  \label{prop:1}
The function $\Phi$ satisfies
\begin{equation}   \label{eq:8}
    \Phi(z,t) = {t}^{-r\nu} \Phi( z {t}^{-\nu},1), \qquad\text{with}\qquad \nu = \frac{\gamma}{r+1} , 
\end{equation}
as well as 
\begin{equation}\label{eq:sym}
	\Phi(z,t) =\omega^{-k}\Phi(\omega^k z,t) ,\quad \text{for}\quad k=0\ldots r, \quad\text{when}\quad \omega\in\mathbb{C}:\quad \omega^{r+1}=1 
\end{equation}
and 
\begin{equation}\label{eq:conj phi}
	\overline{\Phi(\overline{z},t)} = \Phi(z,t). 
\end{equation}

\end{proposition}

\begin{proof}
Use \eqref{eq:7} with $t=1$ to find
\begin{equation}\label{eq Phi1}
    \frac{1}{\Phi(z,1)} =  \Phi^r(z,1) B_r \left( \frac{z}{\Phi(z,1)} \right). 
\end{equation}
Now change $z$ to $zt^{-\nu}$, then
\[    \frac{1}{\Phi(zt^{-\nu},1)} =  \Phi^r(zt^{-\nu},1) B_r \left( \frac{zt^{-\nu}}{\Phi(zt^{-\nu},1)} \right). \]
Multiply both sides of the equation by $t^{r\nu}$ to find
\[ \frac{1}{t^{-r\nu} \Phi(zt^{-\nu},1)} = t^{(r+r^2)\nu} \left( t^{-r\nu} \Phi(zt^{-\nu},1) \right)^r  B_r \left( \frac{zt^{-(r+1)\nu}}{t^{-r\nu} \Phi(zt^{-\nu},1)} \right). \]
Now observe that $(r+1)\nu=\gamma$, so that $t^{-r\nu} \Phi(zt^{-\nu},1)$ satisfies the algebraic equation \eqref{eq:7}. Furthermore it has the same behavior for $z \to \infty$ as
$\Phi(z,t)$ and this proves \eqref{eq:8}.

Besides, if $\omega\in\mathbb{C}$ is the $(r+1)$th root of unity, i.e. $\omega^{r+1}=1$, then \eqref{eq:7} with $z$ replaced by $\omega z$ becomes 
\[
	1= t^{r\gamma} \left(\Phi(\omega z,t)\right)^{r+1} \prod_{k=1}^r \left(\frac{\omega z t^{-\gamma}}{ \Phi(\omega z,t)} -\alpha_k\right) 
\]
which is the same as 
\[
	1= t^{r\gamma} \left(\omega^{-1}\Phi(\omega z,t)\right)^{r+1} \prod_{k=1}^r \left(\frac{z t^{-\gamma}}{\omega^{-1} \Phi(\omega z,t)} -\alpha_k\right) . 
\]
Hence $\omega^{-1}\Phi(\omega z,t)$ satisfies the algebraic equation \eqref{eq:7} and its roots have the same asymptotic behavior for $z\to\infty$ as $\Phi(z,t)$, which implies \eqref{eq:sym} with $k=1$. By iteration, we obtain \eqref{eq:sym}. 

By evoking similar arguments, based on \eqref{eq Phi1}, one shows that \eqref{eq:conj phi} holds. 
\end{proof}

The properties of the algebraic function $\Phi(z,t)$ satisfying \eqref{eq:7} follow directly from the study of  $\Phi(z,1)$ satisfying 
\begin{equation}\label{Phi 1}
	\frac{1}{\Phi(z,1)} 
	= \left(\Phi(z,1)\right)^r \prod_{j=1}^r \left(\frac{z}{\Phi(z,1)} - \alpha_j\right) .
\end{equation}

As a direct consequence of Proposition \ref{prop:1}, one can see from  \eqref{eq:5} that the factors $\phi_k(z,t)$ of $\Phi(z,t)$ satisfy 
\begin{equation*}
	\phi_k(z,t) = t^{-\frac{\gamma}{r+1}} \, \phi_k(z t^{-\frac{\gamma}{r+1}},1),
\end{equation*}
as well as 
\begin{equation*}
	\phi_k(\omega z,t) = \omega^{-1} \phi_k(z,t), \quad\text{for}\quad 1\leq k \leq r.
\end{equation*}

In the light of Proposition \ref{prop:1}, our main interest is the study of the multi-valued algebraic function defined by  
\begin{equation}\label{eq:5 no t}
		1= \left( \Phi(z,1) \right)^{r+1} \prod_{j=1}^r \left(\frac{z}{\Phi(z,1)} -\alpha_j\right) .
\end{equation}
We are primarily interested in the solution  $\Phi(z,1) \sim z^{-r}$ as $z \to \infty$.
We write 
\begin{equation}
\label{Phi z F}
	\Phi(z,1) = z F(z^{r+1})
\end{equation}
and \eqref{eq:5 no t} gives the algebraic equation for the function $F(z)$ in \eqref{eqF}, 
or, equivalently, 
\begin{equation} \label{eqF2}
	\xi = \frac{1}{ F(\xi)  \prod\limits_{j=1}^r \left(1 -\alpha_j F(\xi)\right) } \quad \text{with}\quad  \xi=z^{r+1}. 
\end{equation}
The function $F(\xi)$ is an algebraic function of order $r+1$ and genus $0$. The rational function $\xi(F)$ gives the composition of a conformal map of the sphere to the Riemann surface $\mathcal{R}$ of the function $F(\xi)$ and the projection to the complex plane. 

%
%

If the discriminant has $r+1$ zeros which are the branch points $\xi_0,\ldots \xi_r$, then $F$ has $(r+1)$ branches with the following asymptotic behavior 
\[
	F_0 (\xi) = \frac{1}{\xi} + \mathcal{O} (1/\xi^r) , 
	\quad F_k(\xi) = \frac{1}{\alpha_k}  + \mathcal{O}\left(1/\xi\right), 
	\quad \text{for} \quad k=1,\ldots ,r .
\]

\section{Asymptotic zero distribution}\label{sec:Asympt}

The asymptotic zero distribution of a sequence of polynomials (or the empirical distribution of its zeros) is obtained as a weak$^*$ limit of the zero counting measures. Precisely, the zero counting measure of a polynomial $P$ is defined as 
\[
	\chi (P) := \sum_{\{x:P(x)=0\}} \delta_x
\]
where the sum is taken over all zeros of the polynomial $P$, counting their multiplicities, and $\delta_x$ is the Dirac measure supported at $x$ (i.e., a mass point at $x$). We define the  empirical distribution of the zeros of $Q_n(z,N)$  (its normalized zero counting measure) by
\[
	\mu_{n} = \frac{1}{n} \chi(Q_n(z;N)):= \frac{1}{n} \sum_{k=1}^n  \delta_{y_{{}_{k,n,N}}} , 
\]
where $\{y_{k,n,N}, 1 \leq k \leq n \}$ are the zeros of $Q_n(z,N)$. The corresponding Stieltjes transforms are given by 
\[    \frac{1}{n} \frac{Q_n'(z,N)}{Q_n(z,N)} = \int \frac{1}{z-x} \, d\mu_n(x). 
\]
Our goal is to show that the sequence of measures $(\mu_n)_{n\geq0}$ converges in the weak$^*$ limit to a measure $\mu$.  
For that, we need to find 
\begin{equation}   \label{eq:S}
   \lim_{n \to \infty, n/N \to t} \frac{1}{n} \frac{Q_n'(z,N)}{Q_n(z,N)}. 
\end{equation}
The limit in \eqref{eq:S} will then give the Stieltjes transform
\[  \Gt(z) = \int \frac{d\mu(x)}{z-x}  \] 
of the weak limit of the measures $(\mu_n)_n$ and, by using the Sokhotski-Plemelj formula (Stieltjes inversion formula) one can retrieve from $\Gt$ the measure $\mu$ describing 
the asymptotic distribution of the zeros.

First, we prove the following result, which gives an expression for the Stieltjes transform \(\Gt(z)\) of the measure describing the asymptotic distribution of the the zeros. 

\begin{theorem}  \label{prop:2}
Let \(\Phi(z,t)\) is the solution  of  \eqref{eq:ratio} such that $z^r\Phi(z,t) \to 1$ as $z \to \infty$. 
Then, one has 
\begin{equation}   \label{eq:13b}
   \lim_{N \to \infty} \frac{1}{N} \frac{Q_N'(z,N)}{Q_N(z,N)} = - \frac{1}{r} \int_0^1 \frac{\displaystyle \partial_z \Phi(z,t) }{\Phi(z,t)} \, \mathrm{d}t , 
\end{equation}
uniformly on compact sets of $\mathbb{C} \setminus \Sigma$. 
\end{theorem}

\begin{proof}
Based on  Theorem \ref{Thm1}, the limit \eqref{eq:ratio} exists and the convergence holds uniformly on every compact subset of $\mathbb{C} \setminus \Sigma$, where $\Sigma$ is the set containing all the zeros, which in our case is the
$(r+1)$-star $\bigcup_{j=0}^r \{ x \omega_{r+1}^k, 0 \leq x < \infty \}$.
Taking derivatives in \eqref{eq:ratio} gives
\[   \lim_{n \to \infty, n/N \to t} \frac{Q_n'(z,N)Q_{n+r}(z,N) - Q_n(z,N)Q_{n+r}'(z,N)}{Q_{n+r}^2(z,N)} = \frac{\partial \Phi(z,t)}{\partial z} . \]
The left hand side is equal to
\[    \frac{Q_n(z,N)}{Q_{n+r}(z,N)} \left( \frac{Q_n'(z,N)}{Q_n(z,N)} - \frac{Q_{n+r}'(z,N)}{Q_{n+r}(z,N)} \right), \]
so that
\begin{equation}  \label{eq:10}
   \lim_{n \to \infty, n/N \to t} \left( \frac{Q_n'(z,N)}{Q_n(z,N)} - \frac{Q_{n+r}'(z,N)}{Q_{n+r}(z,N)} \right) = \frac{\displaystyle \frac{\partial \Phi(z,t)}{\partial z} }{\Phi(z,t)} .
\end{equation}
Consider the sum
\begin{equation}   \label{eq:11}
   \sum_{j=0}^{m-1}  \left( \frac{Q_{rj+k}'(z,N)}{Q_{rj+k}(z,N)} - \frac{Q_{r(j+1)+k}'(z,N)}{Q_{r(j+1)+k}(z,N)} \right).
\end{equation}
This telescoping sum equals
\[    \frac{Q_k'(z,N)}{Q_k(z,N)} - \frac{Q_{rm+k}'(z,N)}{Q_{rm+k}(z,N)}.  \]
On the other hand, we can write the sum \eqref{eq:11} also as
\[   N  \sum_{j=0}^{m-1} \int_{\frac{j}N}^{\frac{j+1}N}  \left( \frac{Q_{(\lfloor Nt \rfloor r+k}'(z,N)}{Q_{\lfloor Nt \rfloor r+k}(z,N)} - \frac{Q_{(\lfloor Nt \rfloor +1)r+k}'(z,N)}{Q_{(\lfloor Nt\rfloor +1)r+k}(z,N)} \right)\, \mathrm{d}t, \]
because on $[\frac{j}{N},\frac{j+1}{N})$ we have that $\lfloor Nt \rfloor = j$. Thus, the sum \eqref{eq:11} is also
\[   N \int_0^{m/N}   \left( \frac{Q_{(\lfloor Nt \rfloor r+k}'(z,N)}{Q_{\lfloor Nt \rfloor r+k}(z,N)} - \frac{Q_{(\lfloor Nt \rfloor +1)r+k}'(z,N)}{Q_{(\lfloor Nt \rfloor +1)r+k}(z,N)} \right) \, \mathrm{d}t.  \]
Hence, using \eqref{eq:10} we find
\begin{equation}  \label{eq:12}
  \lim_{m \to \infty, mr/N \to 1} \frac{1}{N} \frac{Q_{mr+k}'(z,N)}{Q_{mr+k}(z,N)} = - \int_0^{1/r} \frac{\displaystyle \partial_z \Phi(z,tr) }{\Phi(z,tr)} \, \mathrm{d}t. 
\end{equation}
By the change of variables $tr \to t$ and $N=mr+k$, we thus have \eqref{eq:13b}.
\end{proof}

From \eqref{eq:8} it follows
\begin{equation*}
	-  \frac{1}{r}\frac{\partial_z \Phi(z,t)}{\Phi(z,t)} 
	 = \frac{1}{z} \left(\frac{t}{r\,\nu} \, \frac{\partial_t \Phi(z,t)}{\Phi(z,t)} + 1 \right) , 
\end{equation*}
and, more relevantly, 
\begin{equation*}
	-  \frac{1}{r}\frac{\partial_z \Phi(z,t)}{\Phi(z,t)} 
	= -  \frac{1}{r} {t}^{- \nu} \left.\frac{\partial_u\Phi( u,1)}{\Phi(u,1)}\right|_{u=z {t}^{-\nu}}
	= -  \frac{1}{r z}  \left.\frac{u \, \partial_u\Phi( u,1)}{\Phi(u,1)}\right|_{u=z {t}^{-\nu}}.
\end{equation*}

\subsection{Proof of Theorem \ref{Thm2}}\label{Subsec Proof Thm2}

Writing $\Phi(u,1) = u F(u^{r+1})$, then 
\[
	\frac{\partial_u \Phi(u,1)}{\Phi(u,1)} = \frac{F(u^{r+1}) +  (r+1) u^{r+1}  F^\prime (u^{r+1})}{uF(u^{r+1}) }
\]
and this gives 
\[
	-  \frac{1}{r}\frac{\partial_z \Phi(z,t)}{\Phi(z,t)} 
	= -  \frac{1}{r\,z}  \left( 1+ (r+1) \left.\frac{ u^{r+1} F'(u^{r+1})}{F(u^{r+1})}\right|_{u=z {t}^{-\nu}} \right), 
\]
where $F'(y) := \frac{\mathrm{d}}{\mathrm{d}y} F(y)$. We write equation \eqref{eqF} as 
\[
	1= y F(y) A_r(F(y))  \qquad \text{with} \qquad A_r(F(y)) =  \prod_{j=1}^r \left(1 -\alpha_j F(y)\right), 
\]
which, after differentiation with respect to $y$, leads to 
\[
	y \frac{F'(y)}{F(y)} = -  \frac{A_r(F(y))}{\left( A_r(F(y)) + F(y) A_r' (F(y))\right)} =  -  \frac{1}{1 + F(y) \, \frac{A_r' (F(y))}{A_r(F(y))}}. 
\]

Hence, based on Theorem \ref{prop:2}, we have  
\begin{align*} 
	\lim_{N \to \infty} \frac{1}{N} \frac{Q_N'(z,N)}{Q_N(z,N)} 
	& =  - \frac{1}{r} \int_0^1 t^{-\nu} \left.\frac{  \partial_u \Phi(u,1) \ }{\Phi(u,1)}\right|_{u=t^{-\nu} z} \, \mathrm{d}t 
\\
	& = - \frac{1}{rz} \int_0^1 \left(1 + (r+1) \frac{z^{r+1} t^{-\gamma} \ F^\prime(z^{r+1} t^{-\gamma})}{F(z^{r+1} t^{-\gamma})} \right) \mathrm{d}t \\
	& = - \frac{1}{rz} \left(1 + (r+1)  \int_0^1 \frac{z^{r+1} t^{-\gamma} \ F^\prime(z^{r+1} t^{-\gamma})}{F(z^{r+1} t^{-\gamma})}\mathrm{d}t  \right) 
\end{align*} 
because 
\begin{align*}
	-  \frac{1}{r}\frac{\partial_z \Phi(z,t)}{\Phi(z,t)} 
	= & -  \frac{1}{r} t^{-\nu}\left.\left(\frac{\partial_u \Phi(u,1)}{\Phi(u,1)} \right)\right|_{u=z {t}^{-\nu}}
	= -  \frac{1}{r} t^{-\nu}\left.\left(\frac{1}{u} + \frac{(r+1)}{u}\frac{ u^{r+1} F'(u^{r+1})}{F(u^{r+1})} \right)\right|_{u=z {t}^{-\nu}}
\end{align*}
and $\nu=\frac{\gamma}{r+1}$. 
With the change of variable  $u=F(z^{r+1} t^{-\gamma})$, the integral becomes 
\begin{align*}
    \lim_{N \to \infty} \frac{1}{N} \frac{Q_N'(z,N)}{Q_N(z,N)} 
    & =  - \frac{1}{rz}\left( 1+ (r+1)\int_0^1 
    \left.\frac{ \xi F'(\xi)}{F(\xi)}\right|_{\xi=z^{r+1} {t}^{-\gamma}}   \mathrm{d}t\right) \\
	& = 	 - \frac{1}{r\,z}  
	\left( 1- \frac{(r+1)}{\gamma} z^{\frac{r+1}{\gamma}} \int_0^{F(z^{r+1})}  u^{\frac{1-\gamma}{\gamma}}\left(\prod_{j=1}^r (1-\alpha_j u)\right) ^{1/\gamma}	\mathrm{d}u  \right), 
\end{align*}
where we have used  \eqref{eqF2} to obtain 
\[
		 t = \left(z^{r+1}u \prod\limits_{j=1}^r \left(1 -\alpha_j u\right) \right)^{1/\gamma}. 
\]

\subsection{Reduction to the real line}
Following Proposition \ref{prop:zeros} and Corollary \ref{cor:zero bounds}, the sequence of polynomials  $(Q_N(z;N))_{N\geq 0}$ is $(r+1)$-fold symmetric, and each polynomial $Q_N(z;N)$ has  $\lfloor N/(r+1)\rfloor$ distinct real zeros $\widetilde{y}_{N,j,\ell}=({y}_{N,j,\ell})^{r+1}$ that can be ordered: 
\[
	0<y_{N,j,1} < y_{N,j,2}  <\ldots < y_{N,j,\lfloor N/(r+1)\rfloor} . 
\]
Hence, there exist $r+1$ polynomial sequences $Q_N^{[j]}(z;N)$, with $j=0,1,\ldots, r$,  such that 
\[
	Q_{N}(z;N) = z^j Q_{\lfloor N/(r+1)\rfloor}^{[j]}(z^{r+1};N)=z^j \prod_{\ell=1}^{\lfloor N/(r+1)\rfloor} \left( z^{r+1} - \widetilde{y}_{N,j,\ell}\right), 
\] 
with  $\lfloor a\rfloor$ denoting the floor function of $a$. 
This means that the Stieltjes transform given above relates to the limiting distribution of the sequence of the $(r+1)$ powers of the zeros $\left(\left\{({y}_{N,j,\ell})^{r+1}\right\}_{\ell=1}^N\right)_{N\geq0}$ as $N\to \infty$. We are also interested in the limiting distribution of the zeros $\left(\left\{ y_{N,j,\ell} \right\}_{\ell=1}^N\right)_{N\geq0}$. The Stieltjes transform, say \(\widehat{\Gt}\), of the limiting zero distribution of the latter  is therefore the pullback of \(\Gt\) under the map \(z\mapsto z^{r+1}\) and they are related by 
\begin{align*}
 	\Gt(z) 
	&= 	(r+1) z^{r}\lim_{N \to \infty} \frac{1}{(r+1)N+k}\  \left.\frac{  \partial_yQ_{N}^{[k]}(y;(r+1)N+k)}{  Q_{ N}^{[k]}(y;(r+1)N+k)}\right|_{y=z^{r+1}} 
	&= (r+1) z^r \widehat{\Gt}(z^{r+1}). 
\end{align*}
Hence, 
\[
	\widehat{\Gt}(z) = \frac{1}{r(r+1)\,z}  
	\left(- 1+  \frac{(r+1)}{\gamma} \frac{1}{\left(F(z) A_r(F(z))\right)^{\frac{1}{\gamma}}} \int_0^{F(z)}  \left(u \, A_r(u)\right) ^{1/\gamma}	\frac{1}{u}\mathrm{d}u  \right).
\]

\subsection{Connection with free probability via $\St$-transform}


The Stieltjes transform of a probability measure $\mu$ on the real line, 
\[    \Gt(z) = \int_\mathbb{R} \frac{d\mu(x)}{z-x},  \]
is analytic on $\mathbb{C}\backslash \mathrm{supp}(\mu)$  where  $\mathrm{supp}(\mu)$ denotes the support of $\mu$, which   is assumed to be a compact set in $\mathbb{C}$. Let 
\[   \Mt(z) = \frac{1}{z} \Gt(1/z) - 1, \]
then $\Mt(z)+1$ is the moment generating function of $\mu$, with 
\[    \Mt(z) = \sum_{n=1}^\infty m_n z^n, \qquad   m_n = \int_\mathbb{R} x^n\, d\mu(x).  \]
If $\Mt^{-1}$ is its functional inverse function, then the $\St$-transform is
\[   \St(u) = \frac{u+1}{u} \Mt^{-1}(u).  \]
The $\St$-transform is an important notion in free probability, see \cite{MingoSpeicher} \cite{Voiculescu}. For instance, the free multiplicative
convolution $\mu_1 \boxtimes\mu_2$ of two measures $\mu_1$ and $\mu_2$ on $\mathbb{R}_+$  satisfies
\[   \St(\mu_1 \boxtimes \mu_2,u) = \St(\mu_1,u) \St(\mu_2,u).  \]
See the Appendix for the $\St$-transform of a few relevant measures.


Our goal is to describe the \(\St\)-transform of the zero limiting distribution of the rescaled polynomial sequence \((Q_n(z))_{n\geq 0}\), based on the corresponding \(\Gt\)-transform given in Theorem \ref{Thm2}. Observe that \eqref{G_integral1} can be expressed as 
\begin{align}\label{zG}
	z\Gt (z)&
	=   g(F(y)) +1 , 
\end{align}
where \(y=z^{r+1}\) and 
\begin{equation}\label{rho}
	g(y)= -\frac{r+1}{r}\left( 1-  \frac{1}{\gamma} \left(\frac{1}{yA_r(y)}\right)^{1/\gamma}
		\int_0^{y} \left(uA_r(u)\right)^{1/\gamma} \frac{1}{u}	\mathrm{d}u  \right), 
\end{equation}
with $A_r(u) = \prod_{j=1}^r (1-\alpha_j u)$.  

\begin{proposition}\label{PropFree} The $\St$-transform of the zero limiting distribution of the rescaled polynomial sequence \((Q_n(z))_{n\geq 0}\) is given by 
\begin{equation}\label{St}
	\St (w)= \frac{w+1}{w}  g^{-1} (w) \ A_r(g^{-1} (w)) , 
\end{equation}
where \(g^{-1}\) denotes the inverse function of  \(g\) in \eqref{rho}. 
\end{proposition}

\begin{proof} From Theorem \ref{Thm2}, we have 
\begin{align*}
	z\Gt (z)-1
	&=  - \frac{r+1}{r}
	\left( 1- \frac{1}{\gamma } \frac{1}{\left(F(y) A_r(F(y))\right)^{\frac{1}{\gamma}}} \int_0^{F(y)}  \left(u \, A_r(u)\right) ^{1/\gamma}	\frac{1}{u}\mathrm{d}u  \right) , 
\end{align*}
which corresponds to \eqref{zG} with $g(y)$ as in \eqref{rho}. 
Consider the change of variable \(z\to 1/z\) to obtain the \(M\)-transform via 
\(\Mt  (y) = \frac{1}{z}\Gt \left(\frac{1}{z}\right) -1\) and this corresponds to
\begin{equation*}
\Mt (y)=  
	 - \frac{r+1}{r}
	\left( 1- \frac{1}{\gamma } \frac{1}{\left(F(1/y) A_r(F(1/y))\right)^{\frac{1}{\gamma}}} \int_0^{F(1/y)}  \left(u \, A_r(u)\right) ^{1/\gamma}	\frac{1}{u}\mathrm{d}u  \right) .
\end{equation*}
Take into account the definition of \(g\) to write 
\(
	\Mt (y) = (g\circ F)(1/y)
\) and therefore its inverse is 
\[
	\Mt ^{-1} (w)= \frac{1}{\left(F^{-1} \circ g^{-1} \right) (w)}. 
\]
Recall that 
\[
	F^{-1} (y) = \frac{1}{yA_r(y)}
\]
to find 
\[
	\Mt ^{-1} (w)=   g^{-1} (w) \ A_r(g^{-1} (w)). 
\]
Use \(\St (w) =  \frac{w+1}{w} \Mt ^{-1}(w) \), to obtain \eqref{St}. 
%
\end{proof}

Note that the expression for the \(\St\)-transform in Proposition \ref{PropFree} does not depend on the algebraic function \(F\). 
When \(\gamma=1\),  
\[
	g(y)= -\frac{r+1}{r}\left( 1- \frac{1}{yA_r(y)}\int_0^{y}  	A_r(u)	\mathrm{d}u  \right), 
\] 
which is a rational function in \(y\). Therefore the \(\Gt\)-transform is a rational function in \(F(y)\), but the \(\St\)-transform might be a challenge to compute, since it involves the computation of the inverse function of the rational function \(g\). 
However, particular choices of \(\alpha_j\) lead to explicit and even rational expressions for the \(S\)-transform. Some examples are worked out in detail in Section \ref{Sec:examples}. When the $\St$-transform is known and is a rational function, then it is possible to relate the asymptotic zero distribution to that of certain  hypergeometric type polynomials, based on the result below. 

\begin{theorem}\normalfont{\cite[Theorem 3.9]{MF-MP2}}\label{Thm MF}  For \(s,t\in\mathbb{Z}_{\geq 0}\), let \(\mathbf{a}_n\in\mathbb{R}^s\) and \(\mathbf{b}_n\in\mathbb{R}^t\) be such that finite limits 
\[\lim_{n\to\infty} \frac{1}{n} \mathbf{a}_n = \tilde{\mathbf{A}}=(\tilde{A}_1,\ldots,\tilde{A}_s)\in\mathbb{R}^s
\quad \text{and} \quad 
\lim_{n\to\infty} \frac{1}{n} \mathbf{b}_n = \tilde{\mathbf{B}}=(\tilde{B}_1,\ldots,\tilde{B}_t)\in\mathbb{R}^t\]
exist. Assume additionally that \(\tilde{A}_j\notin[-1,0)\), \(\tilde{B}_k\neq -1\) and \(\tilde{A}_j\neq\tilde{B}_k\) for all \(j\in\{1,\ldots,s\}\) and \(k\in\{1,\ldots,t\}\). Then any weak$^*$ limit \(\mu\) of the normalized zero-counting measures \(\chi(p_n)\) of the sequence 
\[p_n(x) = \pFq{s+1}{t}{-n,\mathbf{a}_n}{\mathbf{b}_n}{n^{t-s}x}\]
is a positive probability measure compactly supported on \(\mathbb{C}\) for which in a neighborhood of the origin,
the corresponding \(\St\)-transform is given by 
\[
    \St_{\mu}(z) = \frac{\prod\limits_{j=1}^s (z + \tilde{A}_j + 1)}{\prod\limits_{k=1}^t (z + \tilde{B}_k + 1)}\ .
\]
\end{theorem}

We refer to the Appendix for a specific discussion regarding some relevant distributions. 

\section{Examples}\label{Sec:examples}

We start with some degenerate cases, revisiting some known results. 

\subsection{A degenerate case where \texorpdfstring{$\alpha_1=\ldots =\alpha_r$}{alpha1=...=alphar} and \texorpdfstring{$\gamma =r$}{gamma=r}}

We analyze this particular  and degenerate case when $\gamma=r$ and  $\alpha_1=\ldots = \alpha_r=\alpha$. Hence, the recurrence coefficients present the same asymptotic behavior. It is therefore a degenerate case.  The problem here is  similar to the one studied in \cite[Thm 1.2]{Neuschel} and also appeared in the study of Laguerre-Angelesco polynomials in \cite{LeursVA20}.

Under these assumptions \eqref{eq:5 no t}-\eqref{Phi z F} reads as 
\[
	1= \Phi(z) \left( z-\alpha\Phi(z)\right)^r ,
\]
where $\Phi(z):=\Phi(z;1) :=z F(y)$. 
After the substitution, we get 
\[1= y F(y)  \left( 1-\alpha \, F(y)\right)^r .\]
By Theorem \ref{Thm2}, the corresponding  \(\Gt\)-transform given in \eqref{eq:13} (see also   \eqref{zG}-\eqref{rho}) corresponds to  
\begin{equation*}\label{zG case1a}
	z\Gt(z) = g(F(y)) +1= \frac{1}{ 1-\alpha F(y)  } , 
\end{equation*}
since 
\[g(y) = \frac{\alpha y}{1-\alpha y}. \]
Therefore, \(\Gt(z)\) is a solution of the equation 
\[ \left(z\Gt(z)\right)^{r+1} - \frac{y}{\alpha} \left(\left(z\Gt(z)\right) - 1\right)=0. 
\]
By Proposition \ref{PropFree}, the \(\St\)-transform is given by \eqref{St}, where in this case the inverse function of the function $g$ is 
\[g^{-1}(w) = \frac{w}{\alpha (1+w)},\]
and hence 
\[\St(w) = \frac{1}{\alpha}  \frac{1}{\left(w+1\right)^r} .\]
This is the \(\St\)-transform of the Fuss-Catalan distribution, up to a factor \(1/\alpha\), which means
that \(\St\) is the \(\St\)-transform of the Fuss-Catalan distribution on \([0, \alpha(r + 1)^{r+1}/r^r]\).

We take into account \cite[Theorem 3.9 ]{MF-MP2} (see Theorem \ref{Thm MF}) to conclude that the asymptotic zero distribution of $(Q_n(x))_{n\geq0}$ coincides with that of an hypergeometric polynomial 
\[\pFq{1}{r}{-n}{b_1,\ldots,b_r}{n^r x},  \]
where \(\ b_1, \ldots ,\  b_r\) are real constants.

\parbox{0.5\textwidth}{
\includegraphics[width=0.4\textwidth]{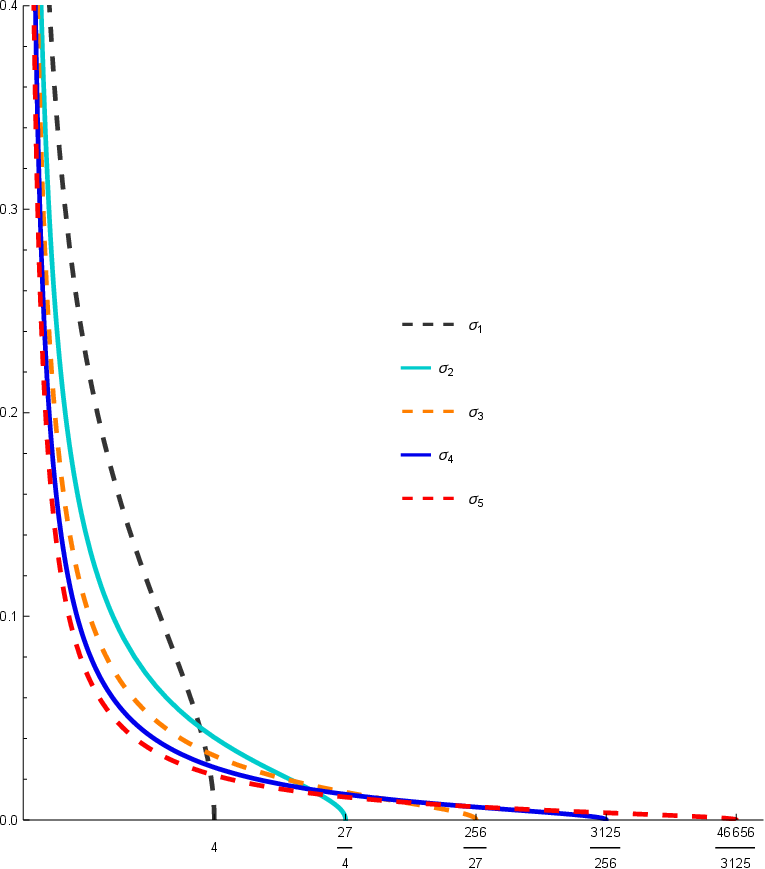}\\
{\footnotesize Plot of the densities $\sigma_j$ for the distribution of the zeros of the rescaled polynomials, under the assumption that $\alpha_j=1$ and $\gamma=r$ for all $j=1,\ldots,r$.}}\quad 
\parbox{0.5\textwidth}{ ~ \\
\includegraphics[width=0.45\textwidth]{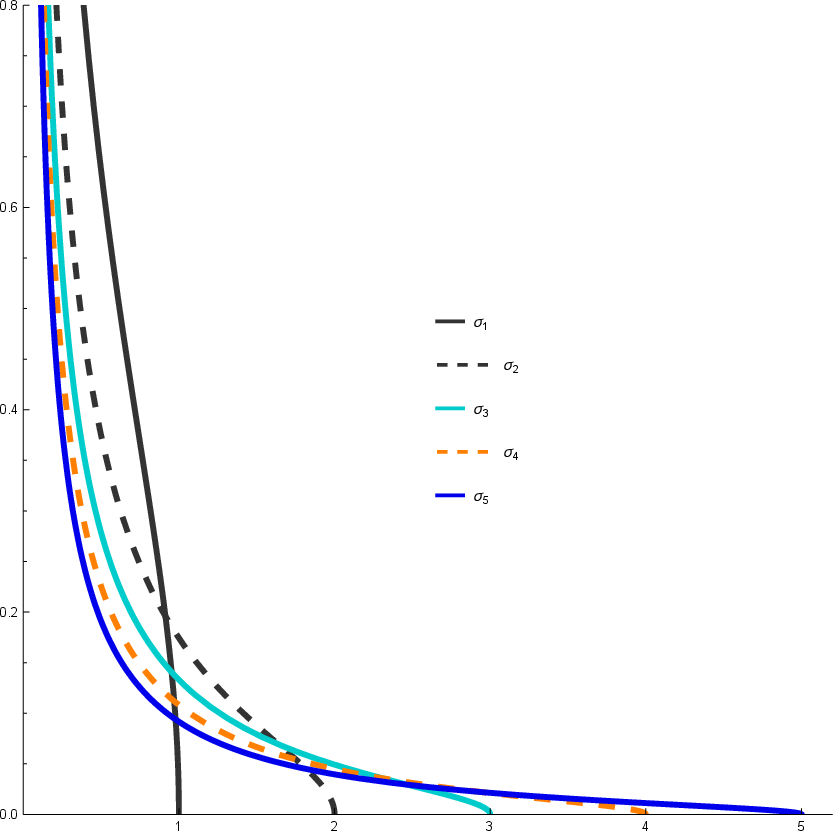} \\
{\footnotesize Plot of the densities $\sigma_j$ for the distribution of the zeros of the rescaled polynomials, under the assumption that $\alpha_j=\left(\frac{r}{r+1}\right)^{r+1}$ and $\gamma=r$ for all $j=1,\ldots,r$.}}


\subsection{A degenerate case where \texorpdfstring{$\alpha_1=\ldots =\alpha_r$}{alpha1=...=alphar} and \texorpdfstring{$\gamma =1$}{gamma=1}}

We analyze another particular case where, like the previous one, all recurrence coefficients present the same asymptotic behavior, that is $\alpha_1=\ldots = \alpha_r=\alpha$, but now we assume $\gamma=1$. We include this example here for a matter of completion. Finding an explicit expression for the \(\St\)-transform is not straightforward as in the previous case, although the corresponding inverse Stieltjes transform presents a similar behavior.

Under these assumptions \eqref{eq:5 no t}-\eqref{Phi z F} reads again as 
\[
	1= \Phi(z) \left( z-\alpha\Phi(z)\right)^r ,
\]
where $\Phi(z):=\Phi(z;1) :=z F(y)$. 
After the substitution, we get 
\[1= y F(y)  \left( 1-\alpha \, F(y)\right)^r .\]
Recall relations \eqref{zG}-\eqref{rho} in Theorem \ref{Thm2} to deduce that \(\Gt\)-transform is given by 
\begin{equation*}\label{zG case1}
	z\Gt(z) = \frac{1}{r\alpha F(y) \left( 1-\alpha F(y)\right)^r } - \frac{1}{r\alpha F(y)}      
		= \frac{y}{r\alpha} - \frac{1}{r\alpha F(y)}.
\end{equation*}
Therefore, \(\Gt(z)\) is a solution of the equation 
\[
	\left(\frac{z^{r+1}-r\alpha z\Gt(z)}{ z^{r+1}-r\alpha z\Gt(z)-\alpha}\right)^{r+1}
	-  \left(\frac{z^{r+1}}{ z^{r+1}-r\alpha z\Gt(z)-\alpha }\right) =0
\]
If we set   
\[  W = W(z):=\frac{z^{r+1}-r\alpha z\Gt(z)}{ z^{r+1}-r\alpha z\Gt(z)-\alpha}
\]
or, equivalently, 
 \[
    \Gt(z)
 =
 \frac{z^{r}}{r\alpha}- \frac{1}{rz} \left(1+\frac{1}{W(z)-1} \right) ,
 \]
then 
\[
	 {W }^{r+1} - \frac{z^{r+1}}{\alpha}\left( W-1\right) =0
\]
holds. We seek a solution of the form \(W=\rho \e^{\ii \theta}\) with \(\theta\) real and \(\rho>0\). We compare real and imaginary parts to obtain the parametric representation 
\[
	\rho = \frac{ \sin( (r+1)\theta) }{\sin( r\theta)}
	\quad \text{and}\quad 
	y=  \frac{\alpha (\sin (r+1)\theta)^{r+1} }{  (\sin r\theta)^{r} \sin(\theta)}, \quad \text{with} \quad \theta\in (0,\pi/(r+1)). 
\]

We compute the inverse Stieltjes transform of $\Gt$ by taking into account that $W_{-}=\rho \e^{-\ii \theta}$ and $W_{+}=\rho \e^{\ii \theta}$ are both solutions of the algebraic equation above. Based on this, we have 
\[
	 \sigma_r(x^{r+1}) =  \frac{1}{2\pi\ii} \lim_{\epsilon\to 0+} \left( \Gt(x+\ii \epsilon) - \Gt(x-\ii \epsilon) \right) 
	= \frac{W_{-} - W_{+}}{2\pi\ii \ r \ z\left|W-1\right|^2} .
\]
Hence, we obtain a representation for the density in parametric form, representing the limiting zero distribution, precisely 
\[\begin{array}{rcl} 
\left[0,\tfrac{\pi}{(r+1)}\right] & \longrightarrow &  \left(0, \frac{(r+1)^{r+1}}{r^r}\right] \times \left[0,+\infty\right)\\[0.3cm]
\theta &\longmapsto & \left( y(\theta) \ ,\ \sigma_r(y(\theta)) \right) 
\end{array}\]
where
\[\sigma_r(y(\theta)) 
= \frac{1}{\pi \, r\,y(\theta)}\frac{\rho(\theta) \sin \theta}{\rho(\theta)^2-2\rho(\theta)\cos\theta+1},
\]
and  
\[\rho(\theta) = \frac{ \sin( (r+1)\theta) }{\sin( r\theta)}
	\quad \text{and}\quad 
y(\theta):=x^{r+1}(\theta)=\frac{\alpha  \  \sin ((r+1)\theta))^{r+1} }{  \sin^r (r\theta)  \   \sin(\theta)}, \quad \text{for} \quad \theta\in (0,\pi/(r+1)).\]

\parbox{0.5\textwidth}{
\includegraphics[width=0.45\textwidth]{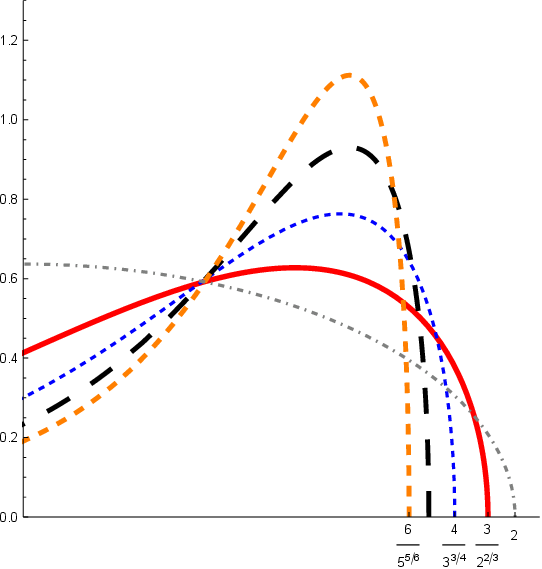}\\
{\footnotesize Plot of the densities for the distribution of the $(r+1)$th principal roots of the zeros of the rescaled polynomials, under the assumption that $\alpha_j=1$ for all $j=1,\ldots,r$.}}\quad 
\parbox{0.5\textwidth}{
\includegraphics[width=0.45\textwidth]{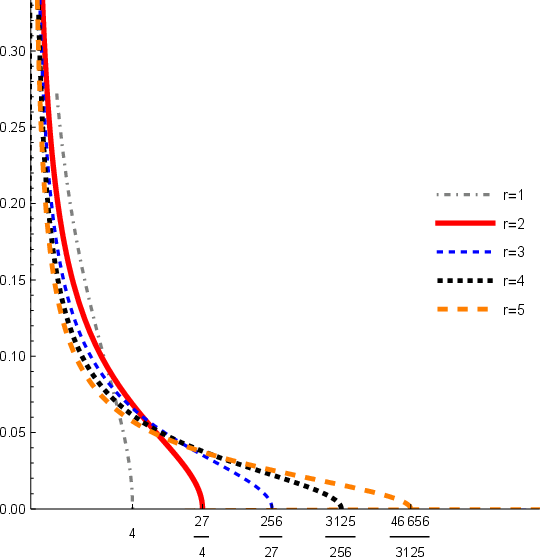} \\
{\footnotesize Plot of the densities for the distribution of the zeros of the rescaled polynomials, under the assumption that $\alpha_j=1$ for all $j=1,\ldots,r$.\\}}

\bigskip

\noindent{\bf The $\St$-transform}

Set $ z\Gt(z)=v+1$ then we have $y=\frac{v+1}{v\St(v)}$. Now recall the equation for $z\Gt(z)$ to conclude that 
\[
	\frac{v+1}{v\St(v)} \left(\frac{v+1}{v\St(v)} - r\alpha (v+1)-\alpha\right)^{r}
	-\alpha  \left(\frac{v+1}{v\St(v)} - r\alpha (v+1)\right)^{r+1} =0. 
\]

We further note that the asymptotic zero densities that we obtained here are closely related to those in \cite[Thm 1.2]{Neuschel}, which describe the asymptotic zero distribution of multiple Laguerre polynomials of type II of the first kind (see \cite[\S 23.4.1]{Ismail}) when the multi-index lies on the diagonal. 

\bigskip

\subsection{Case where \(\alpha_j=\alpha\neq \alpha_1\) for \(j=2,\ldots,r\)}

When \(\alpha_j=\alpha\) for \(j=1,\ldots, r-1\) and \(\alpha_r=(r+1)\alpha\), it follows from Theorem \ref{Thm1} that the ratio asymptotic in \eqref{eq:ratio} is described by the function $\Phi(z):=zF(y)$ satisfying the equation 
\begin{equation}\label{case2: eq for F}
1= y F(y) A_r(F(y)) \quad \text{where} \quad A_r(u) = (1-\alpha u)^{r-1} (1-(r+1) \alpha u).
\end{equation}
If we set  
\[W(y)=\frac{r}{(r+1) (1-\alpha  F)}
\quad \text{or} \quad F(y)=\frac{1}{\alpha }-\frac{r}{\alpha  (r+1) W(y)}, \]
then \eqref{case2: eq for F} implies that $W$ satisfies 
 \[
	W^{r+1} +(r+1) \hat{y}\,W^2 
	-  (2 r+1) \hat{y} \,   W +r  \hat{y} =0. 
 \]
 We want to find the  branch points of the latter algebraic equation in $W$ and these are the zeros of the discriminant $D_{f_r(W)}(\hat{y}) $ in $W$ of the latter equation, which  is the resultant $\mathrm{Res}(f_r'(W),f_r(W))$ between $f_r(W)$ and $f_r'(W):=\partial_W f_r(W)$ where 
$$f_r(W)=  W^{r+1}+ (r+1)  \hat{y} W^2-(2 r+1)  \hat{y} W+r  \hat{y}.$$ 
Since 
$f_r(W) = \frac{1}{r+1} \left(p(W) + W f_r'(W)\right)$, where 
$$p(W)=\left(r^2-1\right)  \hat{y} W^2-r (2 r+1)  \hat{y} W+r (r+1)  \hat{y}, $$
we find
$$
	\mathrm{Res}(f_r'(W),f_r(W)) = \frac{1}{r+1}\mathrm{Res}(f_r'(W),p(W)). 
$$
Note that $\deg p(W)=2$ and has two roots at $W_{-}$ and $W_{+}$. Thus, 
\begin{align*}
	\mathrm{Res}(f_r'(W),p(W))
	&=	(r-1)^r (r+1)^{r} \ \hat{y}^r\ f'\left(\tfrac{2 r^2+r-\sqrt{5 r+4} \sqrt{r}}{2
   \left(r^2-1\right)}\right) f'\left(\tfrac{2 r^2+r+\sqrt{5 r+4}
   \sqrt{r}}{2 \left(r^2-1\right)}\right) 
\end{align*}
and hence 
\[
	D_{f_r(W)}(\hat{y}) = (-1)^{r (r+1)/2} (r-1)^r (r+1)^{r-1} \hat{y}^r \  f_r'\left(\tfrac{2 r^2+r-\sqrt{5 r+4} \sqrt{r}}{2
   \left(r^2-1\right)}\right) f_r'\left(\tfrac{2 r^2+r+\sqrt{5 r+4}
   \sqrt{r}}{2 \left(r^2-1\right)}\right) . 
\]
Note that for $r\geq 2$ the discriminant 
  has a zero of order $r$ at $\hat{y}=0$ and two other zeros at $\hat{y}=\hat{y}_{r-}<0$  and another at $\hat{y}=\hat{y}_{r+}>0$ which are given by 
\begin{align*}
&  \hat{y}_{r,\pm} = 2^{-r} \left(2 r\pm\sqrt{5 r+4}
   \sqrt{r}+1\right) \left(\frac{2 r^2+r\mp \sqrt{5 r+4}
   \sqrt{r}}{r^2-1}\right)^r . 
\end{align*}
%
Hence, $y_{\pm,r}= \left(\frac{r}{r+1}\right)^r \alpha\ \hat{y}_{\pm,r}$ are branch points of $z\Gt(z)$ in equation \eqref{case 2: eq G}: 
\begin{align}
\label{branch pts Case2}
&  {y}_{r,\pm} = 2^{-r} \left(2 r\pm\sqrt{5 r+4}
   \sqrt{r}+1\right) \left(\frac{2 r+1\mp\sqrt{\frac{5 r+4}{r}}}{r-1}\right)^r \alpha,
\end{align}
whose first values are 
 \begin{align}\label{branch pts Case2 table}
 \begin{array}{|l|c|c|ll}
 \hline
& {y}_{r-}& {y}_{r+}\\[0.3cm]\hline
r=2 & \left(5-\frac{7
   \sqrt{7}}{2}\right) \alpha 
   & \left(5+\frac{7
   \sqrt{7}}{2}\right) \alpha  \\[0.3cm]
r=3   &
   \frac{\left(471-133 \sqrt{57}\right)}{72}  \alpha
   &  \frac{\left(471+133 \sqrt{57}\right) }{72} \alpha \\[0.3cm]
r=4  & 
   \left(\frac{129}{16}-\frac{137}{3 \sqrt{6}}\right) \alpha
   & \left(\frac{129}{16} +\frac{137}{3 \sqrt{6}}\right) \alpha\\[0.3cm]
r=5 & \frac{\left(612595-124091 \sqrt{145}\right) \alpha }{64000}& 
   \frac{\left(612595+124091 \sqrt{145}\right) \alpha }{64000} \\[0.3cm] \hline
   \end{array}
 \end{align}

If we  write $W=\rho \e^{\ii \theta}$ and compare the real and imaginary parts we obtain 
\[
\begin{cases}
	\rho^{r+1}\cos(r+1)\theta + (r+1) \hat{y} \rho^2 \cos(2\theta) - (2r+1) \hat{y} \rho\cos\theta +r  \hat{y} =0, \\ 
	\rho^{r+1}\sin(r+1)\theta + (r+1) \hat{y} \rho^2 \sin(2\theta) - (2r+1) \hat{y} \rho\sin\theta  =0, 
\end{cases}
\]
which gives 
 \begin{align*}
 \begin{cases}
 (r+1)  \rho^2  \sin(r-1)\theta - (2r+1)   \rho \sin r\theta+r    \sin(r+1)\theta=0, \\[0.3cm]
 \hat{y} 
	=  \displaystyle \frac{\rho^{r}\sin(r+1)\theta}{ \left(  (2r+1) - 2 (r+1) \rho \cos(\theta)   \right)\sin\theta} . 
	\end{cases}
 \end{align*}
We are interested in the positive solution that vanishes at $\theta=\frac{\pi}{r+1}$, and that is 
\[
	\rho = \frac{ \left(2 (2 r+1) \sin (r\theta  )-\sqrt{-8 r
   (r+1) \cos (2 \theta )-2 \cos (2r \theta  )+2 (2 r+1)^2}\right)}
   {4
   (r+1)\sin(  (r-1)\theta)}. 
\]

Theorem \ref{Thm2} then describes the Stieltjes transform of the limiting zero distribution, which in this case can be written as 
 \[
 	\Gt(z) = \frac{1}{z} \left(  \frac{1}{1-\alpha  (r+1) F(y)}\right), 
 \]
 and thus, it is a solution of the algebraic equation 
\begin{equation}\label{case 2: eq G}
	\alpha  ((r+1) z \Gt(z))^{r+1}  = (r+1) y (z \Gt(z)-1) (r z \Gt(z)+1)^{r-1},
\end{equation}
 precisely, the solution that behaves as $1/z$ as $z\to \infty$. 
  
In order to describe an explicit solution, we first consider the change of variable 
\begin{equation}\label{case 2:W and G}
W=1-\frac{1}{r \left(z \Gt(z)+\frac{1}{r}\right)} 
 \quad \text{or}\quad 
 \Gt(z)
 	= -\frac{1}{rz}\left(\frac{1}{W-1}-1\right)
\end{equation}
 and we set \(\hat{y}= \left(\frac{r}{r+1}\right)^r \, \frac{y}{\alpha }\). 
The density for the limiting zero distribution corresponds to the inverse Stieltjes transform of $\Gt$, that is
\(\sigma:(0,\hat{y}_{r,+}] \to \mathbb{R}\) with 
\[
	\sigma_r(x) = \frac{1}{2\pi\ii} \lim_{\epsilon \to 0+}\left( \Gt(x+\ii \epsilon) -\Gt(x-\ii \epsilon)  \right), 
\]
with $x$ real. 
We take into account 
\eqref{case 2:W and G}, which gives 
\[
 \Gt(z) = -\frac{1}{rz}\left(\frac{\overline{W}-1}{|W-1|^2}-1\right) = \frac{1}{rz} \frac{-\rho\e^{-\ii \theta}}{(\rho^2 -2\rho \cos(\theta)+1) }+\frac{1}{rz}.
\]
Therefore, the limiting zero distribution can be represented in parametric form as follows 
\begin{subequations}
\begin{equation} \label{Case B: density map}
\begin{array}{rcl} 
\left[0,\tfrac{\pi}{(r+1)}\right] & \longrightarrow &  \left(0, \frac{(r+1)^{r+1}}{r^r}\right] \times \left[0,+\infty\right)\\[0.3cm]
\theta &\longmapsto & \left( y(\theta) \ ,\ \sigma_r(y(\theta)) \right) 
\end{array}
\end{equation}
where
\begin{equation}\label{Case B: density}
\sigma_r(y(\theta)) 
= \frac{1}{\pi}\, \frac{\rho (\theta) \sin (\theta )}{r y(\theta) \left(-2 \rho(\theta) 
   \cos (\theta )+\rho(\theta) ^2+1\right)}
\end{equation}
and  
\begin{align} 
& \rho(\theta) =  \frac{ \left(2 (2 r+1) \sin
   ( r\theta  )-\sqrt{-8 r (r+1) \cos (2 \theta )-2
   \cos (2  r\theta  )+2 (2 r+1)^2}\right)}{4
   (r+1)\sin ( (r-1)\theta )} \label{Case B: rho}
\\
&
y(\theta):=x^{r+1}(\theta)=\frac{\alpha  \left(\frac{r+1}{r}\right)^r  \rho(\theta) ^r }{-2 
    (r+1) \cos (\theta ) \rho(\theta)+2 r+1} \frac{\sin ((r+1)\theta  )}{\sin(\theta )}. \label{Case B: y}
\end{align}
\end{subequations}

\begin{figure}[ht!!]
\parbox{0.5\textwidth}{
\includegraphics[width=0.45\textwidth]{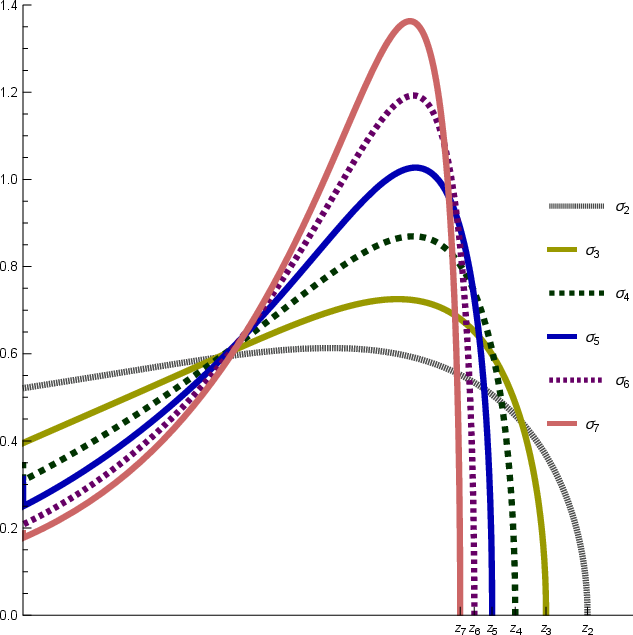}\\
{\footnotesize Plot of the densities for the distribution of the $(r+1)$th principal roots of the zeros of the rescaled polynomials, under the assumption that $\alpha_j=\alpha$ for $j=1,\ldots , r-1$ and $\alpha_r=(r+1)\alpha$.}}\quad 
\parbox{0.5\textwidth}{
\includegraphics[width=0.45\textwidth]{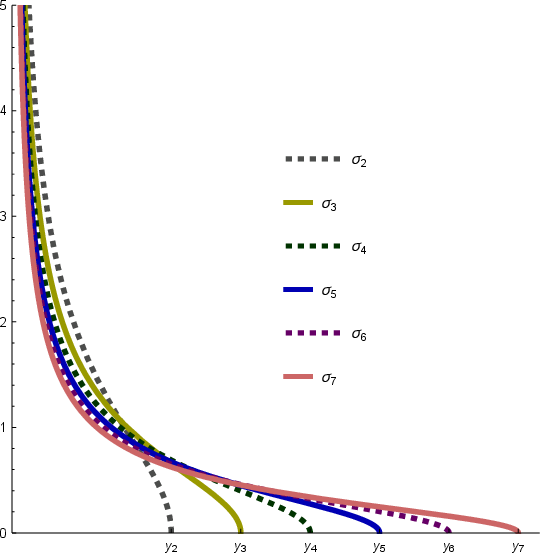} \\
{\footnotesize Plot of the densities  $\sigma_r$ for the distribution of the zeros of the rescaled polynomials for $r=2,\ldots,r$ and $\alpha=r^r (r+1)^{-r}$. Case where $\alpha_j=\alpha$ for $j=1,\ldots , r-1$ and $\alpha_r=(r+1)\alpha$.\\}}
Here $z_r = (y_r)^{\frac{1}{r+1}} $ and 
\[\begin{array}{l@{\qquad}l@{\qquad}l}
y_2=  \frac{\left(10+7 \sqrt{7}\right)}{27} ,
& y_3= 
   \frac{\left(471+133 \sqrt{57}\right)}{972} ,
&y_4= 
   \frac{43}{72}+\frac{137 \sqrt{\frac{2}{3}}}{81},\\[0.4cm] 
y_5= 
   \frac{612595+124091 \sqrt{145}}{864000},
&
   y_6=  \frac{7475156+2651881
   \sqrt{51}}{9112500},
   &y_7=  \frac{104324941+16456635
   \sqrt{273}}{112021056} 
\end{array}\]
\end{figure}

The asymptotic behavior as $\theta$ approaches $0$ or $\frac{\pi}{r+1}$ is given by 
\begin{equation*}
\sigma_r(y(\theta))=
\begin{cases}
a_r\,\theta+\mathcal{O}(\theta^2),
& \theta\to0^+,\\[3mm]
b_r\left(\dfrac{\pi}{r+1}-\theta\right)^{-r}
\left(1+\mathcal{O}\!\left(\dfrac{\pi}{r+1}-\theta\right)\right),
& \displaystyle \theta\to\left(\tfrac{\pi}{r+1}\right)^{-},
\end{cases}
\end{equation*}
where
\begin{align*}
a_r
&=
\frac{r^{r-1}\bigl(2r+1-2(r+1)c_r\bigr)}
{\pi\alpha\,(r+1)^{r+1}c_r^{\,r-1}(1-c_r)^2},
\quad \text{with}\quad 
c_r
=
\frac{r(2r+1)-\sqrt{r(5r+4)}}{2(r^2-1)},
\\[2mm]
b_r
&=
\frac{(2r+1)^r
\sin^{\,r+1}\!\left(\dfrac{\pi}{r+1}\right)}
{\pi\alpha\,(r+1)^{2r}}.
\end{align*}

In order to get the corresponding $\St$-transform associated with this $\Gt$-transform, we write 
\[
	\Gt(z) = \frac{1}{z} \left( H(F(y)) + 1\right), \quad \text{where}\quad H(u) = \frac{1}{1-\alpha  (r+1) u}-1.
\] 
 Based on Proposition \ref{PropFree}, we compute the inverse function 
 \[g^{-1}(w) =\frac{w}{\alpha(r+1)(w+1)} \]
and, hence, we conclude that 
 \[
    \mathcal{S}(w) 
    = \frac{r^{r-1}\left( w+\frac{r+1}{r} \right)^{r-1}}{\alpha  (r+1)^r \left( w+1\right)^r}
    = \frac{r^{r-1}}{\alpha  (r+1)^r}\ 
    \frac{ \left( w+\frac{r+1}{r} \right)^{r-1}}{ \left( w+1\right)^r}. 
 \]

We take into account \cite[Theorem 3.9 ]{MF-MP2} (see Theorem \ref{Thm MF}) to conclude that the asymptotic zero distribution of $(Q_n(x))_{n\geq0}$ coincides with that of an hypergeometric polynomial 
\[\pFq{r}{r}{-n,\frac{n}{r}+a_1,\ldots,\frac{n}{r}+a_{r-1}}{b_1,\ldots, b_r}{nx},  \]
where \(a_1,\ldots,\ a_{r-1} ,\ b_1, \ldots ,\  b_r\) are real constants, and we write 
\[   	Q_n(x) \simeq \frac{(-1)^n(b_1)_n\cdots (b_r)_n}{(\frac{n}{r}+a_1)_n\cdots (\frac{n}{r}+a_r)_n} \pFq{r}{r}{-n,\frac{n}{r}+a_1,\ldots,\frac{n}{r}+a_{r-1}}{b_1,\ldots, b_r}{nx},  \]
meaning that the two sequences share the same asymptotic zero limiting distribution. The two sequences are not the same.

\subsection*{Particular case of  a threefold symmetric polynomial sequence} 

One particular polynomial sequence  satisfying \eqref{eq:1}-\eqref{eq:asper} for \(r=2\) was studied in detail in  \cite[\S 3.1]{AnaWVA}. For that case, the recurrence coefficients are 
 \begin{align*}
	\gamma_{2n}:=\gamma_{2n}(\mu)&=\frac{2}{3 } \frac{(n+1)(2n+1)  }{(3n+\mu+2)} \\
    \gamma_{2n+1}:=\gamma_{2n +1}(\mu) &=  \frac{2}{3 }\frac{(n+1)(2n+3)(n+\mu+1) }{(3n+\mu+2)(3n+\mu+5)}, 
 \end{align*}
 and, thus
 \[
	\gamma_n \sim \frac{2}{3^{j+2}} n\quad \text{if} \quad n=j\mod 2, 
 \]
 with $\mu>-1$. The choice $\mu =1$ gives linear coefficients \(\gamma_{2n}=\frac{2}{9} (2n+1)\) and \( {\gamma_{2n +1} =  \frac{2}{27 }(2n+3)}.\)
 The polynomials satisfying the recurrence relation \eqref{eq:1} are given by 
  \[
 	P_{3n+k}(x;\mu) = \frac{(-1)^n  \left(\mathbf{c}_k\right)_n}{\left(\frac{n}{2}+\frac{a}{3}+b_k\right){}_n}
		\pFq{2}{2}{-n,\frac{n}{2}+\frac{a}{3}+b_k}{\mathbf{c}_k}
   	{x^3} 
	, \quad \text{for} \quad k=0,1,2\ \text{and} \ n\geq 0, 
 \]
 with \(a= \mu +\frac{3(-1)^n+5}{4}\), \(b_0=0,\ b_1=\frac14, \ b_2=1\) and \(\mathbf{c}_1=(\tfrac13,\tfrac23) ,\ \mathbf{c}_2=(\tfrac23,\tfrac43) \) and \(\mathbf{c}_0=(\tfrac43,\tfrac53) \). This is indeed a particular example that falls in the framework of the specification in this section. Particularly, one has that  $\Phi(z,t)= {t}^{-2/3} \Phi( z {t}^{-1/3},1)$ with $\Phi(z,1):=zF(y)$ and $y=z^3$ satisfying the algebraic equation \eqref{case2: eq for F} with $r=2$. 
The algebraic function $F$  lives on a three-sheeted Riemann-surface $\frak{R} = \frak{R}_1 \cup \frak{R}_2 \cup \frak{R}_3$ and
the projections of $F$ to the three sheets $\frak{R}_1, \frak{R}_2, \frak{R}_3$ correspond to three solutions $F_1,F_2,F_3$ of the algebraic equation.
The branch points of the algebraic equation are described in \eqref{branch pts Case2}-\eqref{branch pts Case2 table} with $r=2$. Precisely, in this case, the branch points are
\[   c_0^* =0^2, \quad c_1^* = \frac{10+7\sqrt{7}}{27}, \quad c_2^* =\frac{10-7\sqrt{7}}{27}.  \]
The sheets $\frak{R}_1$ and $\frak{R}_2$ are glued together along the interval $[0,c_1^* ]$, the sheets $\frak{R_2}$ and $\frak{R}_3$ are glued together
along the interval $[c_2^* ,0]$. The three branches of $F$ have the asymptotic behavior
\[  \begin{cases}   F_1(y) =  \mathcal{O}(1/y), \\
                    F_2(y) = \frac92  + {o}(1), \\
                    F_3(y) = \frac{27}{2} + {o}(1), 
    \end{cases}  \quad \text{when} \quad y \to \infty.\]

\begin{figure}[ht!]
\begin{tabular}{lr}
\includegraphics[width=0.4\textwidth]{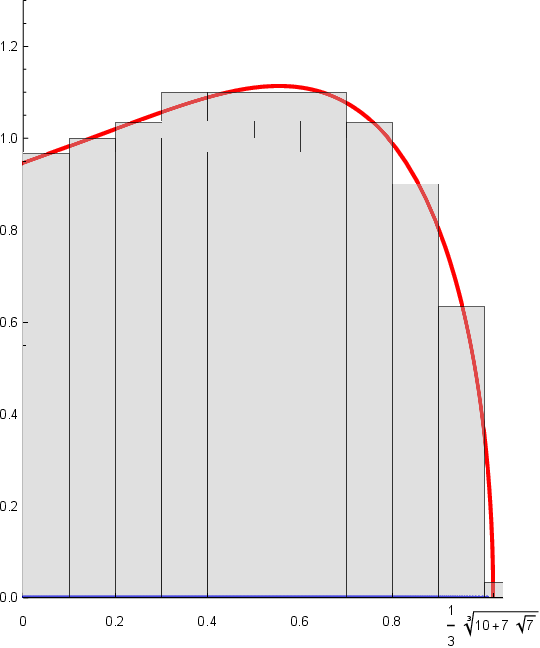}
& \includegraphics[width=0.55\textwidth]{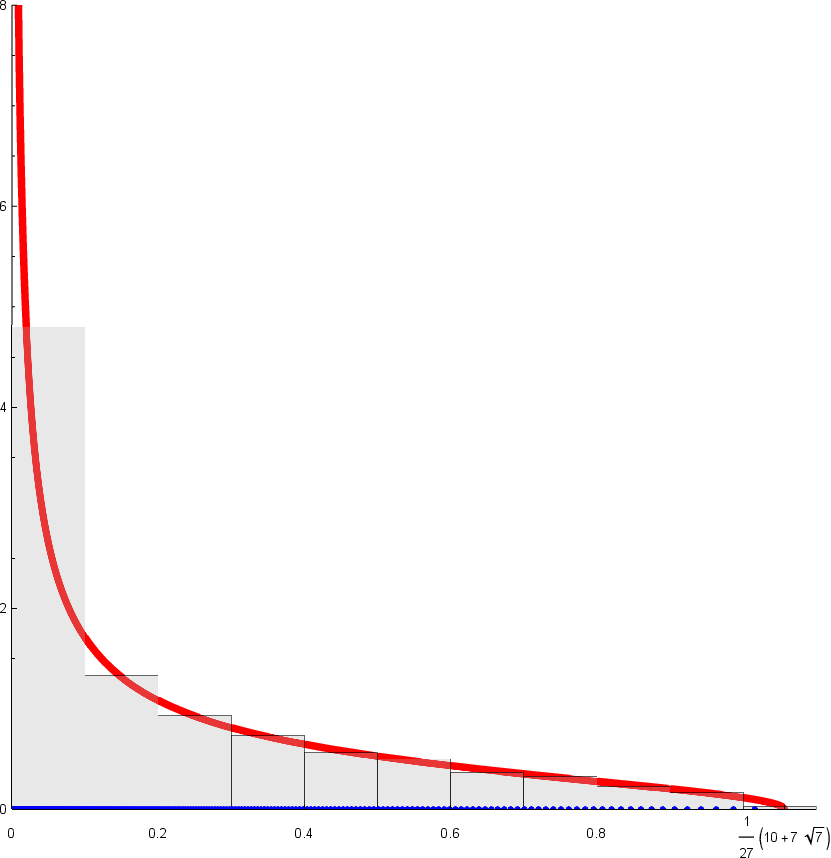}
\end{tabular}
\caption{Density in red of the inverse Stieltjes transform of \(\widehat{\Gt}\) and \(\Gt\) and histogram for the $300$ real roots (on the left) of \(Q_{900}(z)\) and the roots (on the right) of the $300$ positive zeros of \(Q^{[0]}_{300}(z^3):=Q_{900}(z)\) with $\mu=3$. The small blue dots in the horizontal axis represent those roots. }
\end{figure} 
 
 For this concrete example we are able to numerically compute the zeros of the polynomials and analyze the histogram for the zero counting measure. The scaled polynomial $Q_{900}(z;900)$ is a polynomial $Q_{300}^{[0]}(z^3)$ of degree $300$ in $z^3$. For the choice of $\mu=3$, the histogram is plotted below. The red curves correspond to the densities of the cubic roots given in parametric form in \eqref{Case B: density map}-\eqref{Case B: y} with $r=2$ (which are of course independent of the value of $\mu$). An algebraic expression can be obtained by using Cardano's formula for solving the cubic equation for the Stieltjes transform and choosing the correct branch. From here we obtain its inverse transform which leads to the following expression for the density: 
  \[
	\sigma(y) = \frac{7 \sqrt{3} }{12 \pi  }\Big(  f_{+}(y)-f_{-}(y)\Big)
   \left(f_{-}(y)+f_{+}(y)+\tfrac{1}{\sqrt{7}}\right)
 \]
where 
\[
	f_{\pm}(y) =\frac{\sqrt{7}}{7}\sqrt[3]{\frac{9+10 y \pm 3 \sqrt{9+20 y-27 y^2}}{y}}
\]
with $y\in(0,\frac{10+7\sqrt{7}}{27}]$. 

\newpage 

\appendix
\addcontentsline{toc}{section}{Appendix}
\section*{Appendix}


\section{Marchenko-Pastur distribution}
The Marchenko-Pastur distribution $MP(\lambda,\sigma^2)$ describes the distribution of eigenvalues of random matrices from the Wishart ensemble \cite{Marchenko-Pastur}. It is given by the density
\begin{equation}  \label{MP}
   w_{MP}(x) = \frac{1}{2\pi \sigma^2} \frac{\sqrt{(\lambda_+-x)(x-\lambda_-)}}{\lambda x}, \qquad
   \lambda_- \leq x \leq \lambda_+, 
\end{equation}
where $\lambda_{\pm} = \sigma^2(1\pm \sqrt{\lambda})^2$ and $0 <\lambda \leq 1$. If $\lambda >1$ then
one needs to add a dirac measure at the origin and consider $(1-1/\lambda)\delta_0 + w_{MP}$.
This is also the asymptotic distribution of the zeros of Laguerre polynomials $L_n^{(\alpha)}(nx)$
for which the parameter $\alpha=\alpha_n$ grows in such a way that 
$\lim_{n \to \infty} \alpha_n/n = a > -1$ \cite{Gawronski-Shawyer}.
In that case $\sigma^2=1+a$ and $\lambda \sigma^2 = 1$.
The Stieltjes transform is
\begin{equation}   \label{MP-G}
    \Gt(z) = \frac{z-\sigma^2(1-\lambda)-\sqrt{(z-\sigma^2(\lambda+1))^2-4\lambda \sigma^4}}
      {2\lambda z \sigma^2},  
\end{equation}
the $R$-transform is
\begin{equation}  \label{PM-R}
     R(z) = \frac{\sigma^2}{1-\sigma^2 \lambda z}, 
\end{equation}
and the $\St$-transform is
\begin{equation}   \label{MP-S}
      \St(u) = \frac{1}{\sigma^2(1+\lambda u)}.
\end{equation}
                 
\section{Kesten-McKay distribution}
The Kesten-McKay distribution $KM(a,b)$ appears in random walks and random graph theory \cite{Kesten,McKay} but also gives the distribution of eigenvalues of certain random matrices and the asymptotic distribution of the zeros of Jacobi polynomials $P_n^{(\alpha_n,\beta_n)}$ when $\lim_{n \to \infty} \alpha_n/n = a \geq 0$ and $\lim_{n \to \infty} \beta_n/n = b  \geq 0$ \cite{Gawronski}.   We will use the non-symmetric version, with density
\begin{equation}   \label{KM}
    w_{KM}(x) = \frac{2}{C\pi} \frac{\sqrt{(x-\lambda_-)(\lambda_+-x)}}{x(1-x)}, \qquad  \lambda_- \leq x < \lambda_+ , 
\end{equation}
where $0 \leq \lambda_- < \lambda_+ \leq 1$ and
\[   C= (\sqrt{\lambda_-}-\sqrt{\lambda_+})^2+(\sqrt{1-\lambda_-}-\sqrt{1-\lambda_+})^2. \]
For the asymptotic distribution of the zeros of Jacobi polynomials on $[0,1]$, one has
\[      \lambda_{\pm} = \frac{(a+1)(a+b+1)+b+1 \pm 2\sqrt{(a+1)(b+1)(a+b+1)}}{(a+b+2)^2} \]
and $a,b \geq 0$. This gives a mapping between $(\lambda_-,\lambda_+)$ and $(a,b)$ and some of the formulas are more easily expressed in terms of $(a,b)$.
The moments for the (symmetric) Kesten distribution are given in \cite{Hasegawa-Saito}. The Stieltjes transform is
\begin{equation}   \label{KM-G}
   \Gt(z) = \frac{(a+b)z-a - (a+b+2)\sqrt{(z-\lambda_-)(z-\lambda_+)}}{2z(1-z)}, 
 \end{equation}  
and the $\St$-transform is
\begin{equation}  \label{KM-S}
   \St(u) = \frac{u+a+b+2}{u+a+1}.  
\end{equation}

If $a <0$ or $b < 0$ then one has to add dirac distributions at $0$ and/or $1$. For instance
for $0 < \lambda < 1$ the distribution
\[   \lambda \delta_1 + (1-\lambda) KM\left(\frac{a}{1-\lambda},\frac{\lambda}{1-\lambda}\right) \]
has $\St$-transform
\begin{equation} \label{MK-S1}
    \St(u) = \frac{u+2+a-\lambda}{u+1+a}, 
 \end{equation}
 which corresponds to \eqref{KM-S} with $b=-\lambda$. Furthermore the distribution
 \[    \lambda \delta_0 + (1-\lambda) KM\left( \frac{\lambda}{1-\lambda},\frac{b}{1-\lambda}
     \right)  \]
 has $\St$-transform
 \begin{equation}   \label{MK-S0}
     \St(u) = \frac{u+2+b-\lambda}{u+1-\lambda},
  \end{equation}
 which corresponds to \eqref{KM-S} for $a=-\lambda$.    
Finally, for $\lambda_1,\lambda_2 >0$ and $\lambda_1+\lambda_2 < 1$ the distribution
\[  \lambda_1 \delta_0 + \lambda_2 \delta_1 + (1-\lambda_1-\lambda_2) KM \left( \frac{\lambda_1}{1-\lambda_1-\lambda_2}, \frac{\lambda_2}{1-\lambda_1-\lambda_2}\right) \]
has $\St$-transform
\begin{equation} \label{KM-01}
     \St(u) = \frac{u+2-\lambda_1-\lambda_2}{u+1-\lambda_1}, 
\end{equation}
which corresponds to \eqref{KM-S} with $a=-\lambda_1$ and $b=-\lambda_2$.

\section{The Fuss-Catalan distribution}
The Fuss-Catalan numbers are
\[   C_k = \frac{1}{kp+1} \binom{kp+k}{k},    \qquad k \in \mathbb{N}, \]
with $p$ a positive integer. They are the moments of the Fuss-Catalan distribution. 
The density is given by (see, e.g., \cite[Thm. 3.3]{NeuschelPR})
\[  v(x) = \frac{ \sin^2 \varphi \sin^{p-1} p\varphi}{\pi \sin^p (p+1)\varphi}, 
  \qquad 0 < x < \frac{(p+1)^{p+1}}{p^p}, \]
where
\[    x = \frac{ \sin^{p+1} (p+1)\varphi}{\sin \varphi \sin^p p\varphi}.  \]
If $\Gt(z)$ is the Stieltjes transform
of the Fuss-Catalan distribution, then $w=z\Gt(z)$ satisfies the algebraic equation
\[      w^{p+1} - zw + z = 0.   \]
The moment generating function $\Mt(z) = 1/z \Gt(1/z)-1$ hence satisfies the equation
\begin{equation}   \label{FC-M}
       z (\Mt+1)^{p+1} - \Mt = 0, 
\end{equation}
so that the inverse function is
\[     \Mt^{-1}(u) = \frac{u}{(u+1)^{p+1}}, \]
and the $\St$-transform is
\begin{equation}   \label{FC-S}
    \St(u) = \frac{1}{(u+1)^p}.
\end{equation}
Hence the Fuss-Catalan distribution is a $p$-fold free multiplicative convolution of the
Marchenko-Pastur distribution ($\sigma^2=1$, $\lambda=1$).

The Fuss-Catalan distribution appears as the limiting distribution of singular values of a product of Ginibre random matrices \cite{Forrester} and as the limiting distribution of the zeros of multiple Laguerre polynomials \cite{Neuschel}.

\bigskip

\noindent\textbf{Acknowledgment. } We would like to thank the support offered by the London Mathematical Society and the Isaac Newton Institute (INI) through the OPSFOTA network via the INI Network Support for Mathematical Sciences.


\begin{thebibliography}{10}
\bibitem{AKS} A. I. Aptekarev, V. A. Kalyagin, and E. B. Saff, 
\textit{Higher-order three-term recurrences and asymptotics of multiple orthogonal polynomials}, 
Constr. Approx. \textbf{30} (2009), no.~2, 175--223.
\bibitem{AKVI} A.I. Aptekarev, V.A. Kaliaguine,  J. Van Iseghem, 
\textit{The genetic sum's representation for the moments of a system of Stieltjes functions and its application}, Constr. Approx. \textbf{16} (2000), 487--524.
\bibitem{BR08} N. Ben Romdhane, 
\textit{On the zeros of $d$-symmetric $d$-orthogonal polynomials}, 
J. Math. Anal. Appl. \textbf{344} (2008), 888--897.
\bibitem{BleherKuijlaars04} P.M. Bleher, A.B.J. Kuijlaars, 
\textit{Random matrices with external source and multiple orthogonal polynomials},
Int. Math. Research Notices \textbf{2004} (2004), no. 3, 109--129.
\bibitem{DaemsKuijl} E. Daems, A.B.J. Kuijlaars, 
\textit{Multiple orthogonal polynomials of mixed type and non-intersecting Brownian motions},
J. Approx. Theory \textbf{146} (2007), no. 1, 91--114. 
\bibitem{DelLop} S. Delvaux, A. L\'opez, 
\textit{High-order three-term recursions, Riemann-Hilbert minors and Nikishin systems on star-like sets}, 
Constr. Approx. \textbf{37} (2013), 383--453.
\bibitem{DLL} S. Delvaux, A. L\'opez, G. L\'opez Lagomasino,
\textit{On a family of Nikishin systems with periodic recurrence coefficients}, 
(in Russian) Mat. Sb.  \textbf{204}  (2013),  no. 1, 47--78;  translation in 
Sb. Math.  \textbf{204}  (2013),  no. 1--2, 43--74.
\bibitem{DuitsKuijlMo} M. Duits, A.B.J. Kuijlaars, M.Y. Mo,
\textit{The Hermitian two-matrix model with an even quartic potential}, 
Memoirs Amer. Math. Soc. \textbf{217} No. 1022 (2012), vi+105 pp.
\bibitem{Forrester} P.J. Forrester, D-Z. Liu,
\textit{Raney distributions and random matrix theory},
J. Stat. Phys. \textbf{158} (2015), no. 5, 1051--1082.
\bibitem{Gawronski} W. Gawronski,
\textit{Strong asymptotics and the asymptotic zero distributions of Laguerre polynomials $L_n^{(an+\alpha)}$ and Hermite polynomials $H_n^{(an+\alpha)}$},
Analysis \textbf{13} (1993), 29--67.
\bibitem{Gawronski-Shawyer} W. Gawronski, B.L.R. Shawyer,
\textit{Strong asymptotics and the limit distribution of the zeros of Jacobi polynomials $P_n^{(an+\alpha,bn+\beta)}$},
in Progress in Approximation Theory (P. Nevai, A. Pinkus, eds.), Academic Press, New York, 1991, pp.~379--404.
\bibitem{Hasegawa-Saito} T. Hasegawa, S. Saito,
\textit{A note on the moments of the Kesten distribution},
Discrete Mathematics \textbf{344} (2021), nr. 10, 112524.
\bibitem{HJ} R. A. Horn  and C. R. Johnson, 
\textit{Matrix Analysis}, Cambridge University Press, 1985. 
\bibitem{Ismail} M.E.H. Ismail, \textit{Classical and Quantum Orthogonal Polynomials in One Variable},
Encyclopedia of Mathematics and its Applications \textbf{98}, Cambridge University Press,
2005.
\bibitem{Kesten} H. Kesten,
\textit{Symmetric random walks on groups},
Trans. Amer. Math. Soc. \textbf{92} (1959), 336--354.
\bibitem{LeursVA20} M. Leurs, W. Van Assche,
\textit{Laguerre--Angelesco multiple orthogonal polynomials on an $r$-star},
J. Approx. Theory \textbf{250} (2020), 105324.
\bibitem{Abey} A. L\'opez-Garc\'\i a, 
\textit{Asymptotics of multiple orthogonal polynomials for a system of two measures supported on a starlike set}, 
J. Approx. Theory \textbf{163} (2011), no.~9, 1146--1184.
\bibitem{LGLL1} A. L\'opez-Garc\'ia, G. L\'opez Lagomasino, 
\textit{Nikishin systems on star-like sets: ratio asymptotic formulae for the associated multiple orthogonal polynomials}, 
J. Approx. Theory \textbf{225} (2018), 1--40.
\bibitem{LGMD} A. L\'opez-Garc\'ia, E. Mi\~na-D\'iaz,
\textit{Nikishin systems on star-like sets: algebraic properties and weak asymptotics of the associated multiple orthogonal polynomials},
Mat. Sb.  \textbf{209}  (2018),  no. 7, 139--177 (in Russian);
translated in Sb. Math. \textbf{209} (2018), 1051--1088.
\bibitem{AnaWVA} A. Loureiro, W. Van Assche,
\textit{Threefold symmetric Hahn-classical multiple orthogonal polynomials}, Anal. Appl. (Singap.) \textbf{18} (2020), no. 2, 271-332.
\bibitem{Marchenko-Pastur} V.A. Marchenko, L.A. Pastur,
\textit{Distribution of eigenvalues for some sets of random matrices},
Mat. Sb. NS \textbf{72} (114:4) (1967), 507--536 (in Russian);
Math. USSR Sbornik \textbf{1} (1967), Nr.~4, 457--483.
\bibitem{McKay} B.D. McKay,
\textit{The expected eigenvalue distribution of a large regular graph},
Linear Algebra Appl. \textbf{40} (1981), 203--216.
\bibitem{MF-M-P1} A. Mart\'\i nez-Finkelshtein, R. Morales, D. Perales,
\textit{Real roots of hypergeometric polynomials via finite free convolution},
Int. Math. Res. Not. IMRN 2024, no. 16, 11642--11687.
\bibitem{MF-MP2}  A. Mart\'\i nez-Finkelshtein, R. Morales, D. Perales,
\textit{Zeros of generalized hypergeometric polynomials via finite free convolution: applications to multiple orthogonality},
Constr Approx \textbf{63} (2026), no.~1, 153--222. 
\bibitem{MingoSpeicher} J.A. Mingo, R. Speicher,
\textit{Free Probability and Random Matrices},
Fields Institute Monographs, vol. \textbf{35}, Springer, New York, 2017.
\bibitem{NeuschelPR} T. Neuschel,
\textit{Plancherel-Rotach formulae for average characteristic polynomials of products of 
Ginibre random matrices and the Fuss-Catalan distribution},
Random Matrices Theory Appl. \textbf{3} (2014), no.~1, 1450003, 18 pp.
\bibitem{Neuschel} T. Neuschel, W. Van Assche,
\textit{Asymptotic zero distribution of Jacobi-Pi\~neiro and multiple Laguerre polynomials},
J. Approx. Theory \textbf{250} (2016), 114--132. 
\bibitem{WVA} W. Van Assche,
\textit{Ratio asymptotics for multiple orthogonal polynomials}, 
in `Modern Trends in Constructive Function Theory', Contemp. Math. \textbf{661}, Amer. Math. Soc., Providence, RI, 2016, pp.~73--85.
\bibitem{Voiculescu} D.-V. Voiculescu, N. Stammeier, M. Weber,
\textit{Free Probability and Operator Algebras},
M\"unster Lectures in Mathematics, Europ. Math. Soc., 2016.
\end{thebibliography}
\end{document}